\documentclass[11pt]{article}
\usepackage{amsmath, amscd, amssymb, latexsym, epsfig, color, amsthm,enumerate, mathtools, tikz-cd}
\usepackage[all]{xy}
\usepackage{verbatim}
\usepackage{esvect}
\numberwithin{equation}{section}
\def\proof{\smallskip\noindent {\it Proof: \ }}
\def\endproof{\hfill$\square$\medskip}
\newtheorem{theorem}{Theorem}[section]

\newtheorem{corollary}[theorem]{Corollary}
\newtheorem{conjecture}[theorem]{Conjecture}
\newtheorem{lemma}[theorem]{Lemma}

\theoremstyle{definition}
\newtheorem{definition}[theorem]{Definition}

\newtheorem{remark}[theorem]{Remark}

\DeclareMathOperator{\Tot}{\mathrm{Tot}}
\DeclareMathOperator{\skel}{Skel}
\DeclareMathOperator\lk{\mathrm{lk}}

\DeclareMathOperator\st{\mathrm{st}}

\DeclareMathOperator{\Span}{\mathrm{span}}

\DeclareMathOperator{\conv}{\mathrm{conv}}

\DeclareMathOperator{\supp}{\mathrm{supp}}

\DeclareMathOperator{\sign}{\mathrm{sign}}

\newcommand{\field}{{\bf k}}

\newcommand{\R}{{\mathbb R}}
\newcommand{\Q}{{\mathbb Q}}
\newcommand{\Z}{{\mathbb Z}}
\newcommand{\G}{{\mathcal G}}
\newcommand{\I}{{\mathcal I}}
\newcommand{\C}{{\mathcal C}}
\newcommand{\M}{{\mathcal M}}
\newcommand{\V}{{\mathcal V}}

\newcommand{\Stress}{\mathcal S}

\title{Affine stresses: the partition of unity and Kalai's reconstruction conjectures}
\author{
	Isabella Novik\thanks{Research of IN is partially\textsl{} supported by NSF grant DMS-1953815 and by Robert R.~\&  Elaine F.~Phelps Professorship in Mathematics. }\\
	\small Department of Mathematics\\[-0.8ex]
	\small University of Washington\\[-0.8ex]
	\small Seattle, WA 98195-4350, USA\\[-0.8ex]
	\small \texttt{novik@uw.edu}
	\and 
	Hailun Zheng\thanks{Research of HZ is partially supported by a postdoctoral fellowship from ERC grant 716424 - CASe.}\\
		\small Department of Mathematics \& Statistics\\[-0.8ex]
	\small University of Houston-Downtown\\[-0.8ex]
	\small One Main Street, Houston, TX 77002, USA \\[-0.8ex]
	\small \texttt{zhengh@uhd.edu}
}
\begin{document}
\maketitle
\begin{abstract}
Kalai conjectured that if $P$ is a simplicial $d$-polytope that has no missing faces of dimension $d-1$, then the graph of $P$ and the space of affine $2$-stresses of $P$ determine $P$ up to affine equivalence. We propose a higher-dimensional generalization of this conjecture: if $2\leq i\leq d/2$ and $P$ is a simplicial $d$-polytope that has no missing faces of dimension $\geq d-i+1$, then the space of affine $i$-stresses of $P$ determines the space of affine $1$-stresses of $P$. We prove this conjecture for (1) $k$-stacked $d$-polytopes with $2\leq i\leq k\leq d/2-1$, (2) $d$-polytopes that have no missing faces of dimension $\geq d-2i+2$, and (3) flag PL $(d-1)$-spheres with generic embeddings (for all $2\leq i\leq d/2$). We also discuss several related results and conjectures. For instance, we show that if $P$ is a simplicial $d$-polytope that has no missing faces of dimension $\geq d-2i+2$, then the $(i-1)$-skeleton of $P$ and the set of sign vectors of affine $i$-stresses of $P$ determine the combinatorial type of $P$. Along the way, we establish the partition of unity of affine stresses: for any $1\leq i\leq (d-1)/2$, the space of affine $i$-stresses of a simplicial $d$-polytope as well as the space of affine $i$-stresses of a  simplicial $(d-1)$-sphere (with a generic embedding) can be expressed as the sum of affine $i$-stress spaces of vertex stars. This is analogous to Adiprasito's partition of unity of linear stresses for Cohen--Macaulay complexes.
\end{abstract}

{\small \noindent{\bf MSC codes:} 05E45, 13F55, 52B05, 52C25}

	\section{Introduction}
	\subsection{Partition of unity}
	One of the central problems in the theory of face numbers of simplicial complexes is how the information about the local structure of a complex (i.e., properties of the links, or equivalently of the stars) can be used to provide the information about the entire complex. Results of this nature include, among others, McMullen's integral formula that expresses the
	(sums of the) $h$-numbers of a pure simplicial complex in terms of the $h$-numbers of vertex
	links \cite{McMullen70, Swartz2006}, Kalai's observation that a simplicial sphere
	of dimension at least four is stacked if and only if all of its links are stacked \cite{Kalai87}, and Bagchi and Datta's
	$\mu$- and $\sigma$-numbers and their applications \cite{BD14, Murai-15}. The most recent major development on this front is Adiprasito's {\em partition of unity}
	\cite{Adiprasito-g-conjecture, AdiprasitoYashfe} that allows us to express linear stress spaces
	of a Cohen--Macaulay complex (w.r.t.~certain embeddings) as the sums of linear
	stress spaces of vertex stars. This is a fundamental result that has already
	served as an ingredient in several exciting recent breakthroughs, see, for instance, \cite{Adiprasito-toric, Adiprasito-g-conjecture, AdiprasitoYashfe}.
	
	The first goal of this paper is to establish several partition-of-unity-type results for {\em affine} stresses. We defer all definitions to the following sections and for now merely mention that linear stress spaces of a simplicial complex $\Delta$ can be thought of as Weil dual of (the graded components of) an artinian reduction of the Stanley--Reisner ring of $\Delta$. Similarly, affine stress spaces are Weil dual of an artinian reduction modded out by one additional linear form: typically, the sum of the variables. 
	
	Assume $\Delta$ is a $(d-1)$-dimensional simplicial complex (e.g., a simplicial sphere) with vertex set $V$. Specifying an artinian reduction of the Stanley--Reisner ring of $\Delta$ is equivalent to choosing a {\em $d$-embedding} of $\Delta$ --- a map $p$ from $V$ to $\R^d$. There are two most common types of embeddings used in the literature. If $\Delta$ is the boundary complex of a convex polytope $P$, then one can take $p$ to be the natural embedding given by the position vectors of vertices of $P$. The second model is to consider embeddings satisfying certain genericity assumptions. The mildest assumption is to require that the images of vertices of every facet are linearly or affinely independent. This, however, is insufficient in many settings, especially those related to Lefschetz properties. So, instead, one considers a (very) generic embedding, namely, any map such that the multiset of coordinates of vertices is algebraically independent over $\Q$. Both models are extensively used in geometric combinatorics. For instance, both models play a prominent role in the celebrated $g$-theorem  --- the theorem that characterizes $f$-vectors of simplicial spheres. This theorem was first proved for the case of simplicial polytopes by Billera and Lee \cite{BilleraLee} (sufficiency) and Stanley \cite{Stanley80} (necessity); additional proofs of necessity were found by McMullen \cite{McMullen93, McMullen96} and then by Fleming and Karu \cite{Flemingkaru}. A recent breakthrough by Adiprasito \cite{Adiprasito-g-conjecture}, Papadakis and Petrotou \cite{PapadakisPetrotou}, Adiprasito, Papadakis, and Petrotou \cite{AdiprasitoPapadakisPetrotou}, and Karu and Xiao \cite{KaruXiao} settles the case of spheres; all available proofs of this case use generic embeddings.
	
	In this paper, we establish the partition of unity of affine stresses in the following two cases:
	\begin{itemize}
	\item The spaces of affine $1$-stresses of strongly connected simplicial complexes (of dimension $\geq 2$) w.r.t.~embeddings satisfying the property that the images of vertices of any two adjacent facets are affinely independent; see Theorem \ref{lm: decomposition of affine dependency}. This includes the class of normal pseudomanifolds with generic embeddings.
	\item 	The spaces of affine $i$-stresses of simplicial $d$-polytopes with natural embeddings and simplicial $(d-1)$-spheres with generic embeddings (for all $1\leq i\leq (d-1)/2$); see Theorem \ref{lm: general partition of unity, PL} for the case of PL spheres and Theorem \ref{thm: general partition of unity} in the Appendix for all other cases.
	\end{itemize}
	\noindent The proof of the second result is based on the $g$-theorem for polytopes and spheres. 
Specifically, we rely on the fact that certain  Artinian reductions of the Stanley--Reisner rings (over $\R$) of simplicial polytopes and spheres satisfy the hard Lefschetz property; see the end of Section 3.1 for precise statements and references.

	\subsection{Affine stresses and Kalai's conjectures}
We will now describe our main results related to affine stresses and Kalai's reconstruction conjectures. Let $\Delta$ be a simplicial complex with vertex set $V$ and consider the polynomial ring $\R[x_v: v\in V]$ whose variables correspond to vertices of $\Delta$. While we defer most of definitions to later sections, we should mention that spaces of linear and affine stresses of $\Delta$ w.r.t.~a $d$-embedding $p$ are certain homogeneous subspaces of $\R[x_v: v\in V]$.
Specifically, the space of affine $1$-stresses of $\Delta$ consists of linear forms $\sum_{v\in V} a_vx_v$ whose coefficients  $(a_v : v\in V)$ form affine dependencies of the $p$-images of vertices of $\Delta$. This means that when $\Delta$ is the boundary complex of a simplicial $d$-polytope $P$ with its natural embedding, the space of affine $1$-stresses of $\Delta$ contains the same information as the Gale diagram of $P$. Consequently, one may think of the spaces of affine $i$-stresses of polytopes as higher-dimensional analogs of Gale diagrams.

Affine $2$-stresses take their origin in the theory of framework rigidity of graphs. They were extensively used by Kalai \cite{Kalai87} to provide an alternative  proof of the celebrated Lower Bound Theorem of Barnette \cite{Barnette73,Barnette-LBT-pseudomanifolds} as well as to characterize the minimizers. The theory of higher linear and affine stresses was developed in the works of Lee \cite{Lee96} and Tay, White, and Whiteley \cite{Tay-et-al-I,Tay-et-al}. Their main motivation for initiating this theory was the hope of using higher stresses to resolve the $g$-conjecture. This dream was finally realized by Adiprasito in \cite{Adiprasito-g-conjecture}. Other recent applications of linear and affine stresses to the $f$-vector theory and especially to the lower-bound-type questions include \cite{Adiprasito-toric, ANS, KNNZ, NZ-cs-stresses}, to name just a few; see also the results on the $g$-vectors of flag PL spheres in this paper. We encourage the reader to think of further potential applications of spaces of affine stresses. Are there applications similar in spirit to those of usual Gale diagrams?

Motivated by the connection between affine $1$-stresses and Gale diagrams, Kalai proposed the following reconstruction conjectures involving higher affine stresses. Recall that  two simplicial $d$-polytopes $P$ and $Q$ have the same combinatorial type if their boundary complexes are isomorphic; furthermore, $P, Q \subset \R^d$ have the same affine type if there is an invertible affine transformation $\R^d \to \R^d$ that maps $P$ to $Q$.

	\begin{conjecture} \label{conj:comb-type}
	Let $d\geq 4$, $2\leq i\leq d/2$, and let $P$ be a simplicial $d$-polytope. Then the $(i-1)$-skeleton of $P$ (as an abstract simplicial complex) and the space of affine $i$-stresses of $P$ determine the combinatorial type of $P$.
	\end{conjecture}
	\begin{conjecture} \label{conj: generalization of Kalai's conjecture}
	Let $d\geq 4$, let $2\leq i\leq d/2$, and let $P$ be a simplicial $d$-polytope. 
	\begin{enumerate}
	\item If $P$ has no missing $(d-1)$-faces, then the graph of $P$ and the space of affine $2$-stresses of $P$ determine $P$ up to affine equivalence. 
	\item More generally, if $P$ has no missing faces of dimension $\geq d-i+1$, then the space of affine $i$-stresses of $P$ determines $P$ up to affine equivalence.
	\end{enumerate}
	\end{conjecture}
	Conjectures \ref{conj:comb-type} and the first part of Conjecture \ref{conj: generalization of Kalai's conjecture} are due to Kalai: Conjecture \ref{conj:comb-type} was posited in \cite{Kalai-survey} and Conjecture \ref{conj: generalization of Kalai's conjecture}(1) was privately communicated to us and recorded in \cite{N-Z22}. The second part of Conjecture \ref{conj: generalization of Kalai's conjecture} is a generalization of the first part: in addition to being stated for a general $i$, knowing the $(i-1)$-skeleton is not part of the assumptions. (It is still an open problem whether any simplicial $d$-polytope $P$ with no missing faces of dimension $\geq d-i+1$ has the property that every $(i-1)$-face of $P$ participates in an affine $i$-stress on $P$; see Conjecture \ref{conj: the support of stresses} below.)

    That the graph and the space of affine $2$-stresses determine the combinatorial type of a $d$-polytope $P$ (for $d\geq 4$) was verified in \cite{N-Z22}. One does not even need to know the entire space of affine $2$-stresses: knowing the sign vectors of affine $2$-stresses is enough. Furthermore, Cruickshank, Jackson, and Tanigawa \cite[Theorems 1.7 and 8.3]{CJT} recently proved the first part of Conjecture \ref{conj: generalization of Kalai's conjecture} for polytopes whose vertices have {\em generic} coordinates.
	
	Here, we use our partition of unity theorems to establish several results around these conjectures.
	Most notably, we prove the following:
	\begin{itemize}
		\item If $d\geq 4$ and $P$ is any simplicial $d$-polytope that has no missing faces of dimension $\geq d-2$, then the space of affine $2$-stresses of $P$ determines the affine type of $P$; see Theorem \ref{thm: i=2}.
		\item More generally, let $1\leq j< i\leq d/2$.  If $\Delta$ is either 1) a simplicial $d$-polytope that has no missing faces of dimension $\geq d-2i+2$, or 2) a flag PL $(d-1)$-sphere with a generic embedding, then the space of affine $i$-stresses of $\Delta$ determines the space of affine $j$-stresses of $\Delta$; see Theorem \ref{thm: main}.
		
	\item If $2\leq i\leq d/2$ and $P$ is any simplicial $d$-polytope that has no missing faces of dimension $\geq d-2i+2$, then the $(i-1)$-skeleton of $P$ and the set of sign vectors of affine $i$-stresses of $P$ determine the combinatorial type of $P$; see Theorem \ref{thm:comb-type}.

		\item If $1\leq i\leq k\leq d/2-1$, 
		$\Delta$ is a $k$-stacked simplicial $(d-1)$-sphere that has no missing faces of dimension $\geq d-i+1$, and $p$ is a $d$-embedding of $\Delta$ that satisfies certain mild genericity assumptions, then the space of affine $i$-stresses of $\Delta$ determines the space of affine $1$-stresses of $\Delta$; see Theorem \ref{thm: k-stacked polytopes}.
			\end{itemize}
	\noindent It is worth emphasizing that $p$ and $\Delta$ (or $p$ and $P$) are not part of the data. The point of the above results is that as long as $p$ and $\Delta$ satisfy certain assumptions, we can recover the space of affine $1$-stresses of $\Delta$ (and sometimes of affine $j$-stresses for all $1\leq j<i$) solely from the space of affine $i$-stresses of $\Delta$. In particular, in the above cases, knowing  the space of affine $i$-stresses of $(\Delta,p)$ is enough to determine $p$ itself up to an invertible affine transformation.
	
	Similarly to \cite{CJT,N-Z22}, to treat the case of $i=2$, we mainly use the language and tools from the rigidity theory of frameworks. (However, it should be pointed out that \cite{CJT} also employs tools from global rigidity. In particular, the main ingredient in the proof of \cite[Theorem 8.3]{CJT} is an ingeneous extension of the Fogelsanger decomposition theorem to global rigidity; see \cite[Theorem 1.4]{CJT}.)  To treat the case of general $i$, we work with (higher) affine and linear stress spaces. One idea behind the proofs is that if a simplicial $d$-polytope $P$ has no large missing faces and $\tau$ is an $(i-j-1)$-face of $P$, then iteratively taking partial derivatives with respect to the variables corresponding to all vertices of $\tau$ provides a surjection from the space of affine $i$-stresses of $P$ to the space of affine $j$-stresses of the star of $\tau$. The partition of unity then allows us to show that when $P$ has no large missing faces, the space of affine dependencies of vertices of $P$ is determined by the space of affine $i$-stresses of $P$.
	
	In view of our results, it is tempting to posit one more conjecture whose $j=1$ case recovers Conjecture \ref{conj: generalization of Kalai's conjecture}. For a more precise version of this conjecture, see Conjecture \ref{conj: structure of stress spaces}; for an analogous result on linear stresses, see part 2 of Theorem \ref{lm: partition of unity}.
	
	\begin{conjecture}\label{Conj: intro}
		Let $1\leq j<i \leq d/2$. Let $P$ be a simplicial $d$-polytope whose boundary complex has
		no missing faces of dimension $\geq d-i+1$. Then the space of affine $i$-stresses of $P$,
		$\Stress^a_i(P)$, determines the space of affine $j$-stresses of $P$, $\Stress^a_j(P)$. 
	\end{conjecture}

    It is also tempting to posit an analogous conjecture for affine stresses of simplicial spheres w.r.t.~generic embeddings. The restriction on missing faces in Conjectures \ref{conj: generalization of Kalai's conjecture} and \ref{Conj: intro} is unavoidable. Indeed, if $P$ is an $(i-1)$-stacked polytope that is not $(j-1)$-stacked, then $\Stress^a_i(P)$ is the zero space while $\Stress^a_j(P)$ is a non-zero space, and so Conjecture \ref{Conj: intro} does not hold in this case. Similarly, one can slightly perturb the vertices of such $P$ to obtain another polytope $P'$ with the property that $P'$ and $P$ are combinatorially but not affinely equivalent; this is despite the fact that $\Stress^a_{i}(P)=0=\Stress^a_{i}(P')$.
	
\subsection{Organization of the paper}	The rest of the paper is structured as follows. In Section 2, we review several definitions and results related to polytopes and simplicial complexes, such as simplicial spheres and normal pseudomanifods; we also discuss combinatorial properties of these objects. In Section 3, we provide a brief introduction to the theory of stress spaces followed by the discussion of the partition of unity of linear stresses and related results. Then in Section 4 together with the Appendix, we establish the partition of unity of affine stresses for several important classes of complexes, see Theorems \ref{lm: decomposition of affine dependency}, \ref{lm: general partition of unity, PL}, and \ref{thm: general partition of unity}. These tools allow us to prove Kalai's conjectures and their extensions in several cases. Specifically, Section 5 is devoted to reconstructing affine types of complexes that have no missing faces of dimension $\geq d-2$  from the spaces of affine $2$-stresses; see Theorem \ref{thm: i=2}. Sections 6 and 7 focus on reconstructing affine and combinatorial types from higher affine stresses; see Theorems \ref{thm: main}, \ref{thm:comb-type}, and \ref{thm: k-stacked polytopes}. Along the way, we provide various applications of the tools developed in the paper, most notably of Lemma \ref{cor: surjection of affine stress spaces}. One such application is Theorem \ref{thm: flag LB g-number};  it asserts that in the class of flag PL $(d-1)$-spheres, the octahedral sphere has component-wise  minimal $g$-vector.
	
	\section{Preliminaries on polytopes, spheres, and pseudomanifolds}
	A {\em polytope} $P \subseteq \R^d$ is the convex hull of a finite set of points in $\R^d$. The {\em dimension} of $P$ is the dimension of the affine span of $P$. For brevity, we say that $P$ is a {\em $d$-polytope} if $P$ is $d$-dimensional. If the vertices of $P$ are affinely independent, then $P$ is called a (geometric) {\em simplex}. 
	
	Assume $P \subseteq \R^d$ is a $d$-polytope. A hyperplane $H\subseteq \R^d$ is a {\em supporting hyperplane} of $P$ if $P$ is contained in one of the two closed half-spaces determined by $H$. A {\em  (proper) face of $P$} is the intersection of $P$ with any supporting hyperplane of $P$. A face of a polytope is by itself a polytope and each polytope has only finitely many faces. We say that $P$ is {\em simplicial} if all of its (proper) faces are simplices. 
	
	If $v$ is a vertex of $P$, then the {\em vertex figure of $P$ at $v$}, denoted $P/v$,  is the polytope obtained by intersecting $P$ with a hyperplane $H$ that has $v$ on one side and all other vertices of $P$ on the other side. (While the resulting polytope does depend on our choice of $H$, its combinatorial type does not.) In general, if $F$ is a face of $P$, then the {\em quotient of $P$ by $F$}, $P/F$, is obtained from $P$ by iteratively taking vertex figures at the vertices in $F$.
	
	An (abstract) {\em simplicial complex} $\Delta$ with vertex set $V=V(\Delta)$ is a {non-empty} collection of subsets of $V$ that is closed under inclusion and contains all singletons: $\{v\}\in\Delta$ for all $v\in V$. The elements of $\Delta$ are called {\em faces} of $\Delta$. A face $F$ of $\Delta$ is an {\em $i$-face} or a face of {\em dimension $i$} if $|F|=i+1$. For instance, $0$-faces are {\em vertices} and $1$-faces are {\em edges}. (To simplify notation, for faces that are vertices and edges, we write $v$ instead of $\{v\}$ and $uv$ instead of $\{u,v\}$.) The dimension of $\Delta$ is $\max\{\dim F: F\in \Delta\}$. A set $F\subseteq V$ is a {\em missing face} of $\Delta$ if $F$ is not a face of $\Delta$, but every proper subset $\sigma$ of $F$ is a face of $\Delta$. The {\em dimension} of a missing face $F$ is defined as $|F|-1$. A complex $\Delta$ is {\em flag} if all missing faces of $\Delta$ are $1$-dimensional.
	
	Two important examples of simplicial complexes on vertex set $V$ are the (abstract) $(|V|-1)$-simplex $\overline{V} \coloneqq \{\tau : \tau\subseteq V\}$ and the boundary complex of $\overline{V}$, $\partial\overline{V}$. The latter complex consists of all faces of $\overline{V}$ but $V$ itself.
	
	When studying a simplicial complex $\Delta$, one often considers the following subcomplexes of $\Delta$. The subcomplex of $\Delta$ {\em induced by $W\subseteq V(\Delta)$} consists of all faces of $\Delta$ that are contained in $W$. The {\em $i$-skeleton} of $\Delta$, $\skel_i(\Delta)$, is the set of all faces of $\Delta$ of dimension at most $i$. The $1$-skeleton of $\Delta$ is also called the {\em graph} of $\Delta$. If $F$ is a face of $\Delta$, then the {\em antistar} of $F$ is $\Delta-F:=\{\tau\in \Delta: F \not\subseteq \tau\}$. Furthermore, the {\em star of $F$} and the  {\em link of $F$ in $\Delta$} are defined by:
	$$\st(F)=\st(F,\Delta)\coloneqq\{\sigma \in \Delta \ : \  \sigma\cup F\in\Delta\}\text{ and }\lk(F)=\lk(F,\Delta)\coloneqq \{\sigma\in \st(F) \ : \ \sigma\cap F=\emptyset\}.$$  
	
	If $\Gamma$ and $\Delta$ are simplicial complexes on disjoint vertex sets, their \textit{join} is the simplicial complex $\Gamma*\Delta = \{\sigma \cup \tau \ : \ \sigma \in \Gamma \text{ and } \tau \in \Delta\}$. When $\Gamma =\{\emptyset, u\}$ consists of a single vertex, we write $\Gamma*\Delta$ as $u*\Delta$; this complex is the \textit{cone} over $\Delta$ with apex $u$. Thus, for a vertex $v$ of $\Delta$, $\st(v,\Delta)=v*\lk(v,\Delta)$.
	 
	Each simplicial complex $\Delta$ admits a {\em geometric realization} $\|\Delta\|$ that contains a geometric $i$-simplex for each $i$-face of $\Delta$. Conversely, each geometric simplicial complex $D$ corresponds to an abstract simplicial complex whose faces are vertex sets of faces of $D$. For instance,  any simplicial $d$-polytope $P$ gives rise to an abstract simplicial complex $\partial P$ called the {\em boundary complex} of $P$: the faces of $\partial P$ are the vertex sets of all (proper) faces of $P$. With this definition in hand, it is easy to see that for a vertex $v$ of $P$, the boundary complex of $P/v$ is $\lk(v,\partial P)$ and, similarly, for any face $F$, the boundary complex of $P/F$, is $\lk(F,\partial P)$. 
	
	We say that $\Delta$ is a {\em PL $ (d-1)$-sphere} if it is PL homeomorphic to the boundary complex of a $d$-simplex. Similarly,  a {\em PL $d$-ball} is a simplicial complex PL homeomorphic to a $d$-simplex. The PL spheres (balls) belong to a larger class of complexes called {\em simplicial spheres (balls)}: $\Delta$ is a simplicial $(d-1)$-sphere ({\em simplicial $d$-ball}, respectively) if $\|\Delta\|$ is homeomorphic to a $(d-1)$-sphere ($d$-ball, respectively). It is worth noting that while all simplicial $3$-spheres are PL, there are many non-PL $(d-1)$-spheres for $d\geq 6$.
	
	 An even larger class of simplicial complexes is that of homology spheres (homology balls, respectively). Let $\field$ be a field. (We usually consider $\field=\Z/2\Z$ or $\field=\R$.)  A $(d-1)$-dimensional simplicial complex $\Delta$ is a {\em $\field$-homology $(d-1)$-sphere} (or a {\em homology sphere over $\field$}) if for every face $F$ of $\Delta$, including the empty face, the simplicial $\field$-homology of the link of $F$ coincides with that of a $(d-1-|F|)$-sphere. A $d$-dimensional simplicial complex $\Delta$ is a {\em $\field$-homology $d$-ball} if (1) the simplicial $\field$-homology of $\Delta$ coincides with that of a $d$-ball, (2) for every nonempty face $F$ of $\Delta$, the link of $F$ has the $\field$-homology of a $(d-|F|)$-sphere or a $(d-|F|)$-ball, and (3) the {\em boundary complex} of $\Delta$, i.e., the set of all faces whose links have the $\field$-homology of balls, is a $\field$-homology $(d-1)$-sphere. It is important to note that the classes of PL and homology spheres are closed under taking links. On the other hand, the links of simplicial spheres  are homology spheres (over any field) but not necessarily simplicial spheres.
	
	We are now in a position to define an even larger class of normal pseudomanifolds. This requires a bit of preparation.
	A simplicial complex $\Delta$ is called {\em pure} if all {\em facets} (i.e., maximal under inclusion faces) of $\Delta$ have the same dimension.  If $\Delta$ is a pure $(d-1)$-dimensional simplicial complex, then $(d-2)$-faces of $\Delta$ are called {\em ridges}; two facets of such $\Delta$ are {\em adjacent} if they share a common ridge. A pure complex $\Delta$ is {\em strongly connected} if every two facets of $\Delta$ can be connected by a sequence of pairwise adjacent facets of $\Delta$.
	
	Let $\Delta$ be a pure simplicial complex. We say that $\Delta$ is a {\em pseudomanifold without boundary} if every ridge of $\Delta$ is in exactly two facets. Similarly, $\Delta$ is a {\em pseudomanifold with boundary} if every ridge is in at most two facets and there exists a ridge that is contained in only one facet. The {\em boundary of $\Delta$}, $\partial \Delta$, is defined as the subcomplex  of $\Delta$ generated by all the ridges that are contained in only one facet. A pseudomanifold (with or without boundary) is called {\em normal} if the link of every face of codimension at least two is connected. Any normal pseudomanifold (with or without boundary) is strongly connected. 
 Examples of normal pseudomanifolds include homology spheres and balls. Also, antistars of nonempty faces of normal pseudomanifolds without boundary are examples of normal pseudomanifolds with boundary.
	
	Let $\Delta$ be a normal pseudomanifold with boundary. A face $G$ of $\Delta$ is called a {\em minimal interior face} of $\Delta$ if $G\notin\partial\Delta$ but $\partial\overline{G}$ is a subcomplex of $\partial\Delta$. We denote by $\I(\Delta)$ the collection of all minimal interior faces of $\Delta$. We will need the following elementary lemma. 
	
	\begin{lemma} \label{lm:min-int-faces-in-antistars}
		Let $\Delta$ be a normal pseudomanifold without boundary and $F$ a face of $\Delta$. If $\sigma$ is a minimal interior face of $\Delta-F$, then there exists $H\subseteq F$ such that $\sigma\cup H$ is a missing face of $\Delta$.
	\end{lemma}
    \proof Throughout this proof, the links are computed in $\Delta$ and $\Delta$ is suppressed from notation. We write $\sigma$ as $\sigma'' \cup \sigma'$ where $\sigma''=\sigma\cap F$ and $\sigma'=\sigma\backslash \sigma''$. Since $\partial(\Delta-F)=\partial\overline{F}*\lk(F)$ and since $\sigma$ is a minimal interior face of $\Delta-F$, it follows that $\sigma$ has the following properties: (1) $\sigma\in \Delta$ but $\sigma$ is not a subset of $F$ (and so $\sigma'\neq\emptyset$); (2)  $\sigma$ is not a face of $\partial\overline{F}*\lk(F)$, but $\partial\overline{\sigma}$ is a subcomplex of $\partial\overline{F}*\lk(F)$. We conclude that $\sigma'$ is a missing face of $\lk F$. Thus, $\sigma'\cup F$ is not a face of $\Delta$ but $\partial\overline{\sigma'}*\overline{F}$ is a subcomplex of $\Delta$. Now, since $\sigma$ is a face of $\Delta$ but $\sigma\cup F$ is not a face, there must exist a minimal under inclusion subset $H$ of $F\backslash \sigma''$ such that $\sigma\cup H$ is not a face of $\Delta$; in particular, $H$ is nonempty. The set $\sigma\cup H$ is then a desired missing face of $\Delta$.
    \endproof
    
    We close this subsection with a few combinatorial properties related to spheres. Let $\Delta$ be a PL $(d-1)$-sphere. If $\Delta$ contains an induced subcomplex $\overline{A}*\partial \overline{B}$, where $A$ is a $j$-subset of $V(\Delta)$ and $B$ is a $(d-j+1)$-subset of $V(\Delta)$, then we can perform a {\em bistellar flip} on $\Delta$ by replacing $\overline{A}*\partial \overline{B}$ with $\partial \overline{A}*\overline{B}$. The resulting complex $\Delta'$ is another PL $(d-1)$-sphere. We call this operation a {\em $j$-flip.} In particular, the vertex sets of $\Delta$ and $\Delta'$ are identical except in the cases of $j=1$ and $j=d$: in the former case, $\Delta$ has one more vertex (the vertex of $A$) and in the latter case, $\Delta'$ has one more vertex (the vertex of $B$). The following theorem of Pachner \cite{Pachner} gives an alternative definition of PL spheres.
	\begin{theorem}
		Any PL $(d-1)$-sphere can be obtained from the boundary complex of a $d$-simplex by a sequence of bistellar flips.
	\end{theorem}
	A $(d-1)$-dimensional simplicial complex $\Delta$ is {\em shellable} if its facets can be linearly ordered as $F_1, F_2, \dots, F_k$ in such a way that for all $2\leq i\leq k$, the subcomplex $\overline{F_i}\cap (\cup_{j<i}\overline{F_j})$ is pure $(d-2)$-dimensional. Such an ordering of facets is called a {\em shelling} of $\Delta$. Equivalently,  $F_1, F_2, \dots, F_k$ is a shelling of $\Delta$ if for all $i\leq k$,  
	the collection of faces of $\overline{F_i}$ that are not faces of $\cup_{j<i}\overline{F_j}$ has a unique minimal element. This unique minimal face is called the {\em restriction face of $F_i$} and is denoted by $r(F_i)$. We say that $F_i$ is a {\em shelling step of type $m$} if $r(F_i)$ is of size $m$. 
	
	If $\Delta$ is a $(d-1)$-dimensional simplicial complex or a $d$-polytope, we define $f_i(\Delta)$ as the number of $i$-dimensional faces of $\Delta$, where $-1\leq i\leq d-1$. The {\em $h$-numbers} of $\Delta$ are obtained from  the $f$-numbers by the following linear transformation:
	$$h_j(\Delta)=\sum_{i=0}^{j} (-1)^{j-i}\binom{d-i}{d-j}f_{i-1}(\Delta), \quad \text{for all}\ 0\leq j\leq d.$$ We also let $g_0(\Delta)=1$ and $g_j(\Delta)=h_j(\Delta)-h_{j-1}(\Delta)$ for $1\leq j\leq \lceil d/2 \rceil$. The $h$-numbers and $g$-numbers of boundary complexes of simplicial polytopes have various interesting interpretations. For now, we only mention the following classical result; see \cite[Chapter 8]{Ziegler}.
	\begin{theorem}
		The boundary complex of a simplicial polytope is shellable. Furthermore, the $h$-number $h_i$ is exactly the number of shelling steps of type $i$.
	\end{theorem}
	
	\section{Paving the way: the spaces of linear and affine stresses}
	\subsection{Stresses and $h$- and $g$-numbers}
    We start by reviewing several notions and results related to linear and affine stresses. For more details we refer the reader to \cite{Lee94,Lee96} and \cite{Tay-et-al-I, Tay-et-al}.
	
	Let $\Delta$ be a simplicial complex on the vertex set $V=V(\Delta)$. A map $p: V(\Delta) \rightarrow \mathbb{R}^d$ is called a \textit{$d$-embedding} of $\Delta$. In particular, if $\Delta$ is a graph, then $(\Delta, p)$ is called a \textit{$d$-framework}. For $W\subseteq V(\Delta)$, write $p(W)=\{p(v): v\in W\}$ (considered as a multiset if there are repetitions).
	
	Let $X = X(V)=\{x_v : v\in V\}$ be a set of variables with one variable for each vertex and let $\R[X]$ be the polynomial ring over the real numbers in variables $X$. Denote by $\M_i(V)$ the set of all squarefree monomials of degree $i$ in $X(V)$. Each variable $x_v$ acts on $\R[X]$ by $\frac{\partial}{\partial{x_v}}$; for brevity, we will denote this operator by $\partial_{x_v}$. More generally, if $\mu=x_{v_1}\cdots x_{v_s}\in \R[X]$ is a monomial, then define $\partial_\mu : \R[X] \to \R[X]$ by $\rho \mapsto \partial_{x_{v_s}}\cdots\partial_{x_{v_1}}(\rho)$, and if $\ell(X)=\sum_{v\in V} \ell_v x_v$ is a linear form in $\R[X]$, then define
	$$\partial_{\ell(X)} : \R[X]\to\R[X] \quad \mbox{by} \quad \rho \mapsto \sum_{v\in V}\ell_v\cdot\partial_{x_v}\rho=\sum_{v\in V}\ell_v\frac{\partial \rho}{\partial{x_v}}.$$
	
	Given a $d$-embedding $p$ of $\Delta$, consider the $(d+1)\times |V|$ matrix whose columns are labeled by the vertices of $\Delta$ and the column corresponding to  $v\in V$ consists of the vector $p(v)$ augmented by a one in the last position.  The $i$-th row of this matrix, $\boldsymbol\theta_i=[\theta_{iv}]_{v\in V}$, gives rise to a linear form $\theta_i=\sum_{v\in V}\theta_{iv} x_v$. In particular, $\theta_{d+1}=\sum_{v\in V} x_v$. We denote by $\Theta(p)$ or simply by $\Theta$ the sequence $(\theta_1,\ldots,\theta_d, \theta_{d+1})$  of these forms. 
	
	For a monomial $\mu\in \R[X]$, the {\em support} of $\mu$ is $\supp(\mu)=\{v\in V: \;x_v|\mu\}$.
	A homogeneous polynomial $\lambda=\lambda(X)=\sum_\mu \lambda_\mu \mu\in\R[X]$ of degree $k$ is called a {\em linear  $k$-stress} on $(\Delta, p)$ if it satisfies the following conditions:
	\begin{itemize}
		\item Every (non-zero) term $\lambda_\mu \mu$ of $\lambda$ is supported on a face of $\Delta$: $\supp(\mu)\in\Delta$, and
		\item $\partial_{\theta_i}\lambda=0$ for all $i=1,\ldots, d$.
	\end{itemize}
	A linear $k$-stress $\lambda$ on $(\Delta, p)$  that also satisfies $\partial_{\theta_{d+1}}\lambda=0$ is called an {\em affine $k$-stress}. 
	It is immediate from the definitions that the sets of linear $k$-stresses and affine $k$-stresses on $\Delta$ form vector spaces (over $\R$), denoted by $\Stress^\ell_k(\Delta,p)$ and $\Stress^a_k(\Delta, p)$. Furthermore, for all $k$, $\Stress^a_k(\Delta,p)$ is the kernel of $\partial_{\theta_{d+1}} : \Stress^\ell_k(\Delta,p) \to  \Stress^\ell_{k-1}(\Delta,p)$.

	Before proceeding, let us take a moment to discuss the space of affine $1$-stresses in a bit more detail. By definition, an affine $1$-stress on $(\Delta,p)$ is a linear form $a(x)=\sum_{v\in V}a_vx_v$ such that $\partial_{\theta_i} a(x)=0$ for all $i=1,\ldots,d,d+1$. In other words, $a(x)$ is an affine $1$-stress if and only if $\sum_{v\in V} \theta_{iv} a_v=0$ for all $i=1,\ldots, d$ and $\sum_{v\in V}a_v=0$. That is, $a(x)$ is an affine $1$-stress if and only if the vector of coefficients of $a(x)$, $(a_v: v\in V)$, is an affine dependence of the point configuration $\{p(v) : v\in V\}$. This discussion leads to two observations. First, it follows that $\Stress^a_1(\Delta,p)=\Stress^a_1(\Delta,p')$ precisely when $(\Delta,p)$ and $(\Delta,p')$ have the same space of affine dependencies, which happens if and only if $p'$ is obtained from $p$ by an invertible affine transformation. In such a case we say that $(\Delta,p)$ and $(\Delta,p')$ have the {\em same affine type}. Second, since the Gale diagram of a polytope $P\subset \R^d$ is any basis of the space of affine dependencies of the vertex set of $P$, it also follows that when $(\Delta, p)$ is the boundary complex of a simplicial $d$-polytope $P$ with its natural embedding, the space $\Stress^a_1(\Delta,p)$ contains exactly the same information as the Gale diagram of $P$.

	It is known (see \cite{Lee96}) that the dimensions of $\Stress^\ell_k(\Delta,p)$ and  $\Stress^a_k(\Delta, p)$ coincide with the dimensions of the $k$-th graded components of $\R[\Delta]/(\theta_1,\ldots,\theta_d)$ and $\R[\Delta]/(\theta_1,\ldots,\theta_d,\theta_{d+1})$, respectively; here $\R[\Delta]$ is the Stanley--Reisner ring of $\Delta$. In particular, if $\Delta$ is a Cohen--Macaulay complex of dimension $d-1$ (e.g., a simplicial ball or sphere) and $p$ is a $d$-embedding such that for every facet $F$ of $\Delta$, the multiset $p(F)$ is linearly independent, then 
	$\dim \Stress^\ell_k(\Delta,p)=h_k(\Delta)$ for all $0\leq k\leq d$; see \cite{Stanley96}.

	As was mentioned in the introduction, we will mainly work with simplicial polytopes and simplicial spheres (or even homology spheres) using natural embeddings in the former case and generic embeddings in the latter. Specifically, if $P$ is a simplicial $d$-polytope, we let $p$ be the natural $d$-embedding of $\Delta=\partial P$ given by the position vectors of vertices of $P$. 
	If $\Delta$ is a homology $(d-1)$-sphere, then we consider a $d$-embedding $p$ with the property that the multiset of coordinates of the points $p(v)$, $v\in V(\Delta)$, is algebraically independent over $\Q$. Such an embedding is called a \emph{generic} embedding of $\Delta$. The following result is a crucial step in the proof of the $g$-theorem; it provides arguably the most important interpretation of the $g$-numbers of simplicial polytopes \cite{Stanley80, McMullen93, McMullen96, Flemingkaru} and spheres \cite{Adiprasito-g-conjecture, PapadakisPetrotou, AdiprasitoPapadakisPetrotou, KaruXiao}. 
\begin{theorem}\label{thm: Lefschetz}
    Let $(\Delta, p)$ be either the boundary complex of a simplicial $d$-polytope with its natural embedding $p$, or a $\Z/2\Z$-homology $(d-1)$-sphere with a generic embedding $p$, and let $1\leq i\leq \lceil d/2\rceil$. Then $\theta_{d+1}$ is a Lefschetz element, that is, the linear map $\partial_{\theta_{d+1}}: \Stress^\ell_i(\Delta, p) \to \Stress^\ell_{i-1}(\Delta, p)$ is surjective. (In fact, if $d=2i-1$, it is an isomorphism.) In particular, $$\dim \Stress^a_i(\Delta, p)=\dim \Stress^\ell_i(\Delta, p)-\dim \Stress^\ell_{i-1}(\Delta, p)=g_i(\Delta).$$
\end{theorem}

\subsection{The cone lemma and supports of affine stresses}	
Let $\lambda= \sum_\mu \lambda_\mu \mu$ be either a $k$-linear or a $k$-affine stress on $(\Delta,p)$.  We say that $\lambda$ is supported on a subcomplex $\Gamma$ of $\Delta$ if every monomial of $\lambda$ is supported on a face of $\Gamma$. For instance,  $\partial_{x_v}\lambda$ is a $(k-1)$-stress supported on $\st(v,\Delta)$. Let $F\in\Delta$ be a $(k-1)$-face. To simplify notation, we write $x_F\coloneqq\prod_{v\in F} x_v$ and $\lambda_F\coloneqq \lambda_{x_F}$, and call  $\lambda_F$ the {\em weight of $F$ in $\lambda$} or the {\em weight assigned to $F$ by $\lambda$}. If $\lambda_F\neq 0$, we say that $F$ {\em participates in} $\lambda$ or that $F$ is {\em in the support of} $\lambda$.	We refer to $(\lambda_F : F\in \Delta, |F|=k)$ as the {\em squarefree part} of $\lambda$. It is known that a $k$-stress is uniquely determined by its squarefree part, see \cite{Lee96}.

One of very useful results on linear and affine stress spaces is the cone lemma \cite{Lee96, Tay-et-al}. The version below discusses affine stresses.  Let $\Delta$ be any $(d-2)$-dimensional simplicial complex (not necessarily a sphere) with a $(d-1)$-embedding $p'$ and let $\Gamma$ be the cone over $\Delta$ with a $d$-embedding $p$. The $f$-numbers of $\Gamma$ can be easily expressed in terms of the $f$-numbers of $\Delta$; for instance, the $h$- and $g$-vectors of $\Gamma$ and $\Delta$ coincide. This naturally leads to the question of how the stress spaces $\Stress^a_i(\Delta, p')$ and $\Stress^a_i(\Gamma, p)$ are related (for appropriately chosen $p$ and $p'$). The following lemma \cite[Lemma 3.2]{N-Z22}, originally due to Lee (see \cite[Theorem 7]{Lee96}), provides an answer.  

\begin{lemma} \label{cone-lemma1}
	Let $\Delta$ be a simplicial complex with an embedding $p'$. Let $\Gamma=v*\Delta$ be the cone over $\Delta$ with an embedding $p$ such that $p(v)$ is the origin in $\R^d$ and for all $u\in V(\Delta)$, $p(u)=\left[ \begin{array}{cc} a_up'(u) \\ a_u\end{array}\right]$ for some nonzero $a_u\in \R$. Then 
	\begin{enumerate}
		\item there exists an isomorphism  $\phi_k:\Stress^a_k(\Gamma,p)\to \Stress^a_k(\Delta,p')$.
		\item Furthermore, any affine $k$-stress $\omega'$ on $(\Delta,p')$ lifts to an affine $k$-stress $\omega$ on $(\Gamma,p)$ with the property that for every $(k-1)$-face $F\in \Delta$, $\omega'_F=\big(\prod_{u\in F} a_u\big)\omega_F$.
	\end{enumerate} 
	\end{lemma}
    A few remarks are in order. The space of affine stresses is unaffected by Euclidean motions and scalings. Thus, we can always assume that the $p$-image of the cone vertex is the origin. Furthermore, the above lemma applies to vertex links and stars of a simplicial $d$-polytope $P$, $\big((\lk(v, \partial P), p'), (\st(v, \partial P), p)\big)$. Here $p$ is the natural embedding of $\partial P$ and $p'$ is the natural embedding of $\partial(P/v)$ in a hyperplane $H$ that separates $p(v)$ from the rest of the vertices of $P$. Using Euclidean motions and scalings, we can assume that $p(v)$ is the origin and $H$ is given by the equation $x_d=1$. Hence part 2 of Lemma \ref{cone-lemma1} implies that for every $(k-1)$-face $F\in \lk(v, \partial P)$, $\omega'_F$ and $\omega_F$ have the same sign. We refer to \cite[Corollary 3.3]{N-Z22} for a more precise and general statement.
    
    We end this subsection with a conjecture on supports of affine stresses. In \cite{Z-rigidity}, it is shown that if $d\geq 4$, $\Delta$ is a simplicial $(d-1)$-sphere that has no missing faces of dimension $\geq d-1$, and $p$ is a {\em generic} $d$-embedding of $\Delta$,  then every edge participates in some affine $2$-stress on $(\Delta, p)$. Hence in this case, the graph of $\Delta$ is determined by the space of affine 2-stresses. This motivates the following conjecture:
    
    \begin{conjecture}\label{conj: the support of stresses}
    	Let $2\leq i\leq d/2$. Let $\Delta$ be the boundary complex of a simplicial $d$-polytope with its natural embedding $p$, or a simplicial $(d-1)$-sphere (or more generally, a normal $(d-1)$-pseudomanifold without boundary) with a generic embedding $p$. In both cases, assume also that $\Delta$ has no missing faces of dimension $\geq d-i+1$. Then every $(i-1)$-face of $\Delta$ participates in some affine $i$-stress on $(\Delta, p)$. In particular, the space of affine $i$-stresses determines the $(i-1)$-skeleton of $\Delta$.
    \end{conjecture}

\subsection{The partition of unity of linear stresses}
Our original motivation for this paper came in part from the following result on linear stresses of spheres. Recall that $\M_i(V)$ is the set of all squarefree monomials of degree $i$ in $X(V)$.

\begin{theorem}\label{lm: partition of unity}
		Let $\Delta$ be a homology $(d-1)$-sphere (over some field) and let $p$ be a $d$-embedding of $\Delta$ such that the $p$-images of vertices of any facet $G\in \Delta$ are linearly independent. Then the following holds.
		\begin{enumerate}
		\item For all $1\leq i\leq d-1$ and $k\leq d-i$,
		$$\Stress^\ell_i(\Delta, p)=\sum_{v\in V(\Delta)} \Stress^\ell_i(\st(v), p) = \sum_{F\in \Delta, \, |F|=k} \Stress^\ell_i(\st(F), p).$$
		\item Furthermore, for all $1\leq j<i\leq d$, 
		\begin{eqnarray*}\Stress^\ell_j(\Delta,p) &=&\Span\big\{\partial_{\mu}\omega \, : \, \omega\in \Stress^\ell_i(\Delta,p), \, \mu\in  \M_{i-j}(V(\Delta))\big\}\\
		&=&  \Span\big\{\partial_{x_F}\omega \, : \, \omega\in \Stress^\ell_i(\Delta,p), \, F\in \Delta, \, |F|=i-j\big\}.
		\end{eqnarray*}
	\end{enumerate}
	\end{theorem}

For shellable spheres, this result is due to Lee \cite[Theorem 16]{Lee96}: Lee only proved part 2 of the statement, but since for a stress $\omega$ on the entire complex, $\partial_{x_F}\omega$ is a stress supported on the star of $F$ (it is $0$ if $F$ is not a face), part 1 is an immediate consequence of part 2. For general homology spheres (in fact, for general Cohen--Macaulay complexes), part 1 was proved by Adiprasito \cite[Lemma 3.4]{Adiprasito-g-conjecture}, see also 
 \cite{Adiprasito-toric, AdiprasitoYashfe}. In words, part 1 asserts that any linear $i$-stress on $\Delta$ can be written as the sum of linear $i$-stresses supported on the stars. This property is known as the {\em partition of unity of linear stresses}. Since we could not find the proof of part 2 in the literature, we provide it here for completeness.

\smallskip\noindent {\it Proof of part 2: \ }
Let $\Delta$ be a homology $(d-1)$-sphere and let $F$ be an $(i-j-1)$-face of $\Delta$. Then
$\Stress^\ell_i(\Delta-F, p)$ is a subspace of $\Stress^\ell_i(\Delta,p)$, and this subspace is precisely the kernel of the map $\partial_{x_F}: \Stress^\ell_i(\Delta,p) \to \Stress^\ell_{j}(\st(F),p)$. Hence we have the following exact sequence: 
		$$0\to \Stress^\ell_i(\Delta-F,p) \to \Stress^\ell_i(\Delta,p) \stackrel{\partial_{x_F}}{\longrightarrow} \Stress^\ell_{j}(\st(F),p).$$
		Since $\Delta-F$, $\Delta$, and $\st(F)$ are $(d-1)$-dimensional Cohen--Macaulay complexes, 
		the dimensions of the three spaces in the sequence are $h_i(\Delta-F)$, $h_i(\Delta)$, and $h_{j}(\st(F))=h_{j}(\lk(F))$, respectively. 
		Since for every $k$, $f_{k-1}(\Delta)=f_{k-1}(\Delta-F)+f_{k-(i-j)-1}(\lk(F))$, it follows easily that
		$h_i(\Delta)=h_i(\Delta-F)+h_{j}(\lk (F))$ (cf.~\cite[Lemma 4.1]{Athan}). We conclude that the right-most map in this sequence, $\partial_{x_F}: \Stress^\ell_i(\Delta,p)\to \Stress^\ell_{j}(\st(F),p)$, must be onto. Thus
		\[
		\Span\big\{\partial_{x_F}\omega: \omega\in \Stress^\ell_i(\Delta,p), \, F\in \Delta, \, |F|=i-j\} \big\} = \sum_{F\in \Delta, \, |F|=i-j} \Stress^\ell_{j}(\st(F),p)=\Stress^\ell_{j}(\Delta,p),
		\]
		where the last step is by part 1. The result follows.
		\endproof

 The cone lemma yields the following variation of Theorem \ref{lm: partition of unity}. We will use it in Section 7.

\begin{corollary} \label{cor:simplex-join-sphere}
Let $\Delta$ be a homology $(k-1)$-sphere, let $\overline{F}$ be a $(d-k-1)$-simplex, and let $\Gamma=\overline{F}*\Delta$. Let $p$ be a $d$-embedding of $\Gamma$ such that  $p(H)$ is linearly independent for every facet $H\in \Gamma$. Then for all $1\leq i\leq k-1$ and $t\leq k-i$,
$\Stress^\ell_i(\Gamma, p)= \sum_{G\in \Delta, \, |G|=t} \Stress^\ell_i(\st(G, \Gamma), p).$
Furthermore, for all $1\leq j<i\leq k$, 
		 $\Stress^\ell_j(\Gamma,p)=\Span\big\{\partial_{\mu}\omega: \omega\in \Stress^\ell_i(\Gamma,p), \,  \mu\in\M_{i-j}(V(\Delta))\big\}.$
\end{corollary}
\proof Theorem \ref{lm: partition of unity} and the cone lemma for linear stresses \cite{Lee96} imply that
$$\Stress^\ell_i(\Gamma,p)=\Stress^\ell_i(\overline{F}*\Delta,p)=\sum_{G\in \Delta, \, |G|=t} \Stress^\ell_i(\overline{F}*\st(G, \Delta), p)=\sum_{G\in \Delta, \, |G|=t} \Stress^\ell_i(\st(G, \Gamma), p).$$
For the second statement, observe that if $G\in\Delta$, then $\Gamma-G=\overline{F}*(\Delta-G)$ is a $(d-1)$-dimensional Cohen--Macaulay complex. The rest of the proof is identical to that of part 2 of Theorem \ref{lm: partition of unity}.
\endproof

Part 2 of Theorem \ref{lm: partition of unity} provides a structural result on spaces of linear stresses of spheres. This result along with Conjectures \ref{conj: generalization of Kalai's conjecture} and \ref{Conj: intro} suggest that an analogous statement might hold for affine stresses. To this end, we propose the following conjecture.

\begin{conjecture}\label{conj: structure of stress spaces}
Let $(\Delta, p)$ be either the boundary complex of a simplicial $d$-polytope with its natural embedding $p$, or a $\Z/2\Z$-homology $(d-1)$-sphere with a generic embedding $p$. If $\Delta$
		 has no missing faces of dimension $\geq d-i+1$, then for all $1\leq j<i\leq d/2$, $$\Stress^a_j(\Delta,p)=\Span\big\{\partial_{\mu}\omega: \omega\in \Stress^a_i(\Delta,p), \, \mu\in \M_{i-j}(V(\Delta))\big\}.$$
\end{conjecture}

\section{The partition of unity of affine stresses}
	
The goal of this section is to establish two versions of the partition of unity of affine stresses. Since we only work with the stress spaces on $(\Delta,p)$ and never explicitly use the associated $\Theta$ (as in Section 3), to avoid confusion, from now on, we will write $c$ instead of $\theta_{d+1}$ to denote $\sum_{v\in V(\Delta)} x_v$.
	
\begin{theorem}\label{lm: decomposition of affine dependency}
	Let $d\geq 3$. Let $\Delta$ be a strongly connected $(d-1)$-dimensional complex with an embedding $p$ in $\R^d$ such that the $p$-images of vertices of any two adjacent facets are affinely independent. Then $\Stress^a_1(\Delta,p)=\sum_{G\in \Delta, \dim G=d-3} \Stress^a_1(\st(G),p)$. In particular, $\Stress^a_1(\Delta,p)=\sum_{v\in V(\Delta)} \Stress^a_1(\st(v),p)$.
\end{theorem}
\begin{proof}
	If $\Delta$ has at most $d+1$ vertices, then $\Stress^a_1(\Delta, p)$ and $\Stress^a_1(\st(G), p)$ are the zero spaces for all faces $G\in \Delta$, and the claim holds. Thus, assume that $\Delta$ has at least $d+2$ vertices, and hence, at least three facets. Pick an ordering $F_1,F_2,\dots,F_m$ of the facets of $\Delta$ that satisfies the following property: for every $i\geq 1$, $F_{i+1}$ is adjacent to at least one facet with a smaller index. Such an ordering exists because $\Delta$ is strongly connected. Label the vertices of $F_1\cup F_2$ by $v_1,\dots,v_{d+1}$. Let $3\leq i_1<i_2<\cdots<i_{g_1(\Delta)}\leq m$ be the indices such that for all $1\leq j\leq g_1(\Delta)$, the size of $V(\cup_{k=1}^{i_j-1}F_k)$ is smaller than the size of $V(\cup_{k=1}^{i_j}F_k)$; label the new vertex introduced at step $i_j$ by $v_{d+1+j}$. 
	
Let $1\leq j\leq g_1(\Delta)$.	By the defining property of our ordering, there exists $t<i_j$ such that $F_t$ is adjacent to $F_{i_j}$. If $t=1$, we take $s=2$ and notice that $F_s=F_2$ is adjacent to $F_t=F_1$. If $t>1$, then again by the defining property of our ordering, there exists $s<t$ such that $F_s$ and $F_t$ are adjacent. In either case, the complex $\overline{F_s}\cup\overline{F_t}\cup\overline{F_{i_j}}$ contains $d+2$ vertices, among them $v_{d+1+j}\in F_{i_j}\backslash(F_s\cup F_t)$. Thus $(\overline{F_s}\cup\overline{F_t}\cup\overline{F_{i_j}}, p)$ supports a non-trivial affine $1$-stress $\omega_j$. Also, since $F_s$ and $F_t$ are adjacent, the $p$-images of the $d+1$ vertices of $F_s\cup F_t$ are affinely independent. It follows that $\omega_j$ assigns a nonzero weight to $v_{d+1+j}$ and zero weights to all $v_k$ for $k>d+1+j$. Therefore, the affine 1-stresses $\omega_1,\ldots,\omega_{g_1(\Delta)}$ are linearly independent, and hence span $\Stress^a_1(\Delta)$. The result follows since $\omega_j$ is supported on the star of $G_j\coloneqq F_s\cap F_t\cap F_{i_j}$ which is a face of dimension $\geq d-3$. 
\end{proof}

	Applying the above theorem to the boundary complex of a simplicial polytope $P$ with its natural embedding, 
	we obtain the following
	\begin{corollary}\label{cor: decomposition of affine dependency}

		Let $d\geq 3$ and let $\Delta$ be the boundary complex of a simplicial $d$-polytope $P$ with its natural embedding $p$. Then 
$$\Stress^a_1(\Delta,p)=\sum_{\stackrel{G\in \Delta}{\dim G=d-3}} \Stress^a_1(\st(G),p), \quad\text{and hence}, \quad\Stress^a_1(\Delta,p)=\sum_{v\in V(\Delta)} \Stress^a_1(\st(v),p).$$
	\end{corollary}

	\begin{remark} \label{rm: linear}
	Let $\Delta$ be a normal $(d-1)$-pseudomanifold with boundary and let $p$ be a $d$-embedding of $\Delta$ such that the $p$-images of vertices of any facet of $\Delta$ are linearly independent. Assume also that $\Delta$ is not a simplex. Then using the same ideas as in the proof of Theorem \ref{lm: decomposition of affine dependency}, one easily shows that $\Stress^\ell_1(\Delta,p)=\sum_{R} \Stress^\ell_1(\st(R),p)$, where the sum is over interior ridges of $\Delta$. Since every interior ridge contains a minimal interior face, we also obtain that $\Stress^\ell_1(\Delta,p)=\sum_{F\in \I(\Delta)} \Stress^\ell_1(\st(F),p)$.
	\end{remark}

	Our second result is about the partition of unity of higher affine stresses.
	\begin{theorem}\label{lm: general partition of unity, PL}
		Let  $2\leq i \leq (d-1)/2$ 
		and let $\Delta$ be a PL $(d-1)$-sphere with a generic $d$-embedding $p$. Then $\Stress^a_i(\Delta, p)=\sum_{v\in V(\Delta)}\Stress^a_i(\st(v), p)$.
	\end{theorem}
	\begin{proof}
		Recall that by Pachner's theorem, any PL $(d-1)$-sphere $\Delta$ can be obtained from the boundary complex of a $d$-simplex by a sequence of bistellar flips. Hence, to prove our statement, it suffices to show that if $\Delta'$ is obtained from $\Delta$ by a $j$-flip (for any $1\leq j\leq d$) and $\Delta$ satisfies the statement of the theorem, then so does $\Delta'$. For the rest of the proof, we let $p$ be a generic embedding of $V(\Delta\cup\Delta')$ and we suppress $p$ from our notation.
		
		By Theorem 13 in \cite{Lee96}, for $1\leq k\leq d$,
		$$\begin{cases}
			\Stress^\ell_k(\Delta')\oplus\Stress^\ell_k(\partial(\overline{A\cup B}))=\Stress^\ell_k(\Delta)& \text{if }j<d-j+1 \quad\text{and}\quad j \leq k\leq d-j, \\
			\Stress^\ell_k(\Delta')=\Stress^\ell_k(\partial(\overline{A\cup B}))\oplus\Stress^\ell_k(\Delta) & \text{if }j>d-j+1 \quad\text{and}\quad d-j+1 \leq k\leq j-1, \\
			\Stress^\ell_k(\Delta')= \Stress^\ell_k(\Delta) & \text{otherwise}. \\
		\end{cases}$$
		For a simplicial complex $\Gamma$ with a face $G$ and a vertex $u$, we write $\omega_{G, \Gamma}$ to denote an affine $i$-stress on $\Gamma$ with $G$ in the support and we write $\omega_{u, G, \Gamma}$ to denote an affine $i$-stress on $\st(u, \Gamma)$ with $G$ in the support. Recall our convention that $c\coloneqq\sum_v x_v$. Since $d\geq 2i$ and since $\Delta$ and $\Delta'$ are PL $(d-1)$-spheres with generic embeddings, Theorem \ref{thm: Lefschetz} implies that the following maps are surjective: $$\Stress^\ell_i(\Delta) \stackrel{\partial c}{\longrightarrow} \Stress^\ell_{i-1}(\Delta),\quad \Stress^\ell_i(\Delta') \stackrel{\partial c}{\longrightarrow} \Stress^\ell_{i-1}(\Delta').$$
		Furthermore, the space $\Stress^\ell_m(\partial(\overline{A\cup B}))$ is $1$-dimensional for all $0\leq m \leq d$, and so the map $\partial c: \Stress^\ell_i(\partial(\overline{A\cup B})) \to \Stress^\ell_{i-1}(\partial(\overline{A\cup B}))$ is an isomorphism. Consequently,
		$$ \begin{cases}
			\Stress^a_i(\Delta')\oplus \Span\{\omega_{A, \Delta}\}=\Stress^a_i(\Delta) & \text{if }i=j,\\
			\Stress^a_i(\Delta')=\Stress^a_i(\Delta)\oplus \Span\{\omega_{B, \Delta'}\} & \text{if }i=d-j+1, \\
			\Stress^a_i(\Delta')=\Stress^a_i(\Delta) & \text{otherwise}. \\
		\end{cases}$$
		
		For a vertex $v\in A$ and $j\neq 1$, $\lk(v, \Delta')$ is obtained from $\lk(v, \Delta)$ by a $(j-1)$-flip $ (\overline{A\backslash v})*\partial \overline{B}\mapsto \partial(\overline{A\backslash v})*\overline{B}$. For a vertex $v\in B$ and $j\neq d$, $\lk(v, \Delta')$ is obtained from $\lk(v, \Delta)$ by a $j$-flip $ \overline{A}*\partial(\overline{B\backslash v})\mapsto \partial \overline{A}* (\overline{B\backslash v})$. Finally for every $v\notin A\cup B$, $\lk(v, \Delta)=\lk(v, \Delta')$. Since vertex links of $\Delta$ and $\Delta'$ are PL $(d-2)$-spheres and $d-1\geq 2i$, the same argument as above (combined with the cone lemma) implies that  
		$$ \begin{cases}
			\Stress^a_i(\st(v,\Delta'))\oplus \Span \{\omega_{v, A\backslash v,\Delta}\}=\Stress^a_i(\st(v,\Delta)) & \text{if }v\in A \quad \text{and}\quad  i=j-1,\\
			\Stress^a_i(\st(v,\Delta'))\oplus \Span \{\omega_{v, A,\Delta}\}=\Stress^a_i(\st(v,\Delta)) & \text{if }v\in B \quad \text{and} \quad i=j,\\
			\Stress^a_i(\st(v,\Delta'))=\Stress^a_i(\st(v,\Delta))\oplus \Span\{\omega_{v, B,\Delta'}\} & \text{if }v\in A \quad \text{and}\quad  i=d-j+1, \\
			\Stress^a_i(\st(v,\Delta'))=\Stress^a_i(\st(v,\Delta))\oplus \Span\{\omega_{v, B\backslash v,\Delta'}\} & \text{if }v\in B \quad \text{and}\quad i=d-j, \\
			\Stress^a_i(\st(v,\Delta')) =\Stress^a_i(\st(v,\Delta)) & \text{otherwise}. \\
		\end{cases}$$
		
		Assume that $A=\{u_1, \dots, u_j\}$ and $B=\{v_1, \dots, v_{d-j+1}\}$. Since the spaces of affine $i$-stresses ($i\geq 2$) of the entire complex and of the vertex stars are not affected by a facet subdivision  or its inverse, we further assume that $j\neq 1$ and $j\neq d$, and so $V(\Delta)=V(\Delta')$. We consider the following five cases:
		
		{\bf Case 1: $i\neq j-1, j, d-j, d-j+1$.} In this case, by our assumptions on $\Delta$,
		$$\Stress^a_i(\Delta')=\Stress^a_i(\Delta)=\sum_{v\in V(\Delta)}\Stress^a_i(\st(v, \Delta))=\sum_{v\in V(\Delta)}\Stress^a_i(\st(v, \Delta')).$$
		
		{\bf Case 2: $i=j$.} To start, we claim that for any $2\leq k\leq d-i+1$, $$\Stress^a_i(\st(v_1,\Delta))+\Stress^a_i(\st(v_k, \Delta))=\Stress^a_i(\st(v_1,\Delta))+\Stress^a_i(\st(v_k, \Delta')).$$
		Recall that for $1\leq k\leq d-i+1$, $\Stress^a_i(\st(v_k, \Delta'))\oplus \Span\{\omega_k\}=\Stress^a_i(\st(v_k,\Delta))$, where $\omega_k$ is any affine $i$-stress on $\st(v_k,\Delta)$ with $A$ in the support. We will now show that for any fixed $2\leq k\leq d-i+1$, we can choose $\omega_1$ and $\omega_k$ to be the same stress. If $d\geq 2i+2$ and $k\geq 2$, then $\lk(v_1v_k,\Delta)$ and $\lk(v_1v_k,\Delta')$ are spheres of dimension $d-3\geq 2i-1$ and since their $(i-1)$-skeleta defer only in $\{A\}$, that is, $\skel_{i-1}(\lk(v_1v_k,\Delta))=\skel_{i-1}(\lk(v_1v_k,\Delta')\cup\{A\})$, Theorem \ref{thm: Lefschetz} (combined with the cone lemma) implies that there exists an affine $i$-stress $\omega$ on $\st(v_1v_k,\Delta)$ that has $A$ in its support. We take $\omega_1$ and $\omega_k$ to be that $\omega$.
		Similarly, if $d=2i+1$, then $\lk(v_1v_k,\Delta)\cup \{B\backslash v_1v_k\}$ and $\lk(v_1v_k, \Delta')\cup \{A\}$ share the same $(i-1)$-skeleton. Hence, by applying Theorem \ref{thm: Lefschetz}, we conclude that there is a nonzero element $\omega$ of  $\Stress^a_i\big(\st(v_1v_k,\Delta)\cup\{B\backslash v_1v_k\}\big) = \Stress^a_i\big(\st(v_1v_k,\Delta')\cup\{A\}\big)$ that has $A$ in its support. Since both $\st(v_1, \Delta)$ and $\st(v_k,\Delta)$ contain the subcomplex $\st(v_1v_k,\Delta)\cup \{B\backslash v_1v_k\}$, we can again take $\omega_1$ and $\omega_k$ to be that $\omega$. Hence 
		\begin{equation*}
			\begin{split}
				\quad  &\Stress^a_i(\st(v_1,\Delta))+\Stress^a_i(\st(v_k, \Delta))\\
				=&\Big(\Stress^a_i(\st(v_1,\Delta'))\oplus \Span \{\omega\}\Big)+\Big(\Stress^a_i(\st(v_k, \Delta'))\oplus\Span\{\omega\}\Big)=\Stress^a_i(\st(v_1,\Delta))+\Stress^a_i(\st(v_k, \Delta')),
			\end{split}
		\end{equation*}
		as desired. Since the above equation holds for all  $2\leq k \leq d-i+1$, we infer that 
		\begin{equation*}
			\begin{split}
				\sum_{k=1}^{d-i+1}\Stress^a_i(\st(v_k,\Delta))=&
				\sum_{k=2}^{d-i+1} \left[\Stress^a_i(\st(v_1,\Delta))+\Stress^a_i(\st(v_k,\Delta))\right]=
				\sum_{k=2}^{d-i+1} \left[\Stress^a_i(\st(v_1,\Delta))+\Stress^a_i(\st(v_k,\Delta'))\right]\\
				=& \ \Stress^a_i(\st(v_1,\Delta))+\sum_{k=2}^{d-i+1} \Stress^a_i(\st(v_k,\Delta')).
			\end{split}
		\end{equation*}
		
		On the other hand, since for all $z\notin B$, $\Stress^a_i(\st(z, \Delta')=\Stress^a_i(\st(z, \Delta))$, we also obtain that 
		\begin{equation*}
			\begin{split}
				& \Stress^a_i(\Delta')\oplus\Span\{\omega_{A, \Delta}\}=\Stress^a_i(\Delta)=\sum_{z\in V(\Delta)} \Stress^a_i(\st(z,\Delta))\\
				=\ & \Stress^a_i(\st(v_1,\Delta))+\sum_{z\in V(\Delta), \, z\neq v_1} \Stress^a_i(\st(z,\Delta'))=\left( \sum_{z\in V(\Delta)} \Stress^a_i(\st(z,\Delta'))\right) \oplus\Span\{\omega_1\}.
			\end{split}
		\end{equation*}
		Since $\omega_{A, \Delta}$ is any affine $i$-stress on $\Delta$ with $A$ in the support, we can take $\omega_{A, \Delta}=\omega_1$. Finally, since $\Stress^a_i(\Delta')\supseteq \sum_{z\in V(\Delta)} \Stress^a_i(\st(z,\Delta'))$, the above equation yields that $\Stress^a_i(\Delta')=\sum_{z\in V(\Delta)} \Stress^a_i(\st(z, \Delta'))$, and so $\Delta'$ satisfies the partition of unity.
		
		{\bf Case 3: $i=d-j+1$.} Then
		\begin{equation*}
			\begin{split}
				&\Stress^a_i(\Delta')=\Stress^a_i(\Delta)+\Span\{\omega_{u_1, B, \Delta'}, \dots, \omega_{u_{d-i+1}, B, \Delta'}\} \\
				= &\sum_{k=1}^{d-i+1}\left( \Stress^a_i(\st(u_k,\Delta)) +\Span\{\omega_{u_k, B, \Delta'}\}\right) + \sum_{z\in V(\Delta)\backslash A} \Stress^a_i(\st(z,\Delta))
				=\sum_{z\in V(\Delta)}\Stress^a_i(\st(z,\Delta')).
			\end{split}
		\end{equation*}
		
		{\bf Case 4: $i=j-1$.} In this case $\skel_{i-1}(\Delta')=\skel_{i-1}(\Delta)$ and hence $\Stress^a_i(\Delta')=\Stress^a_i(\Delta)$. For $1\leq k\leq i+1$, $\Stress^a_i(\st(u_k,\Delta'))\oplus \Span\{\omega_{u_k, A\backslash u_k, \Delta}\}=\Stress^a_i(\st(u_k,\Delta))$. The subcomplex $\st(u_kv_1,\Delta') \cup \{A\backslash u_k\}$ supports a nontrivial affine $i$-stress $\omega_k$ with $A\backslash u_k$ in its support. Since $\st(u_kv_1,\Delta') \cup \{A\backslash u_k\}$ is contained in both $\st(u_k, \Delta)$ and
		$\st(v_1,\Delta')$, it follows that $$\Stress^a_i(\st(u_k,\Delta))\subseteq \Stress^a_i(\st(u_k,\Delta'))+\Stress^a_i(\st(v_1,\Delta')).$$
		
		On the other hand, $\Stress^a_i(\st(z, \Delta'))=\Stress^a_i(\st(z,\Delta))$ for $z\notin A\cup B$. Furthermore, $\Stress^a_i(\st(v_k, \Delta))\subseteq \Stress^a_i(\st(v_k,\Delta'))$ for all $1\leq k\leq d-i$ and $d\geq 2i+1$. Hence we conclude that $$\Stress^a_i(\Delta')=\Stress^a_i(\Delta)=\sum_{z\in V(\Delta)} \Stress^a_i(\st(z, \Delta))\subseteq \sum_{z\in V(\Delta')} \Stress^a_i(\st(z, \Delta')).$$ 
		The partition of unity in this case follows since $\Stress^a_i(\Delta')\supseteq \sum_{z\in V(\Delta')} \Stress^a_i(\st(z, \Delta'))$ always holds.
		
		{\bf Case 5: $i=d-j$.} The proof is similar to case 4. The only difference is that we switch the roles of $u_k$ and $v_k$, and obtain that $\Stress^a_i(\st(u_k, \Delta))\subseteq \Stress^a_i(\st(u_k,\Delta'))$ for $1\leq k\leq j=d-i$,  and $\Stress^a_i(\st(v_k, \Delta))\subseteq \Stress^a_i(\st(v_k,\Delta'))+\Stress^a_i(\st(u_1, \Delta'))$ for $1\leq k\leq d-j+1$.  
	\end{proof} 
	
	In view of Theorem \ref{lm: general partition of unity, PL}, we ask if the partition of unity of affine stresses also holds for simplicial polytopes with natural embeddings and (non-PL) simplicial spheres with generic embeddings.
	
	\begin{conjecture}\label{conj: general partition of unity}
		Let $2 \leq i\leq (d-1)/2$. Let $(\Delta, p)$ be either the boundary complex of a simplicial $d$-polytope with its natural embedding $p$, or a $\Z/2\Z$-homology $(d-1)$-sphere with a generic embedding $p$. Then $\Stress^a_i(\Delta,p)=\sum_{v\in V(\Delta)} \Stress^a_i(\st(v), p)$.
	\end{conjecture}

	\begin{remark}
After seeing the previous version of this paper, Adiprasito, Murai, and the anonymous referee commented that Conjecture \ref{conj: general partition of unity} can be proved by adapting the spectral sequence argument of the proof of Theorem 33 from \cite{AdiprasitoYashfe}. We feel that this proof is beyond the scope of this paper, and so instead of presenting it here, we sketch the main steps and ideas in the Appendix. 
\end{remark}

\section{A warm-up: reconstructing from affine 2-stresses}
	In this section we prove the following theorem (cf.~part 1 of Conjecture \ref{conj: generalization of Kalai's conjecture}). 
	\begin{theorem}\label{thm: i=2}
		Let $d\geq 4$. Let $\Delta$ be either (i) a normal $(d-1)$-pseudomanifold without boundary with a generic embedding $p$, or (ii) the boundary complex $\partial P$ of a simplicial d-polytope $P$ with its natural embedding $p$. In both cases, assume also that $\Delta$ has no missing faces of dimension $\geq d-2$. Then $\Span\{\partial_{x_v} \omega: \omega\in\Stress^a_2(\Delta,p), \, v\in V(\Delta)\}=\Stress^a_1(\Delta,p)$, and so the space $\Stress^a_2(\Delta, p)$ determines the space $\Stress^a_1(\Delta, p)$. In particular,  $\Stress^a_2(\Delta,p)$ determines the positions of vertices of $\Delta$ up to affine equivalence.
	\end{theorem}
	We remark that for the case of normal pseudomanifolds with generic embeddings, Theorem 8.3 in \cite{CJT} establishes a more general result: it shows that $\Stress^a_2(\Delta,p)$ determines $\Stress^a_1(\Delta,p)$ as long as $\Delta$ has no missing faces of dimension $d-1$. The result on polytopes with natural embeddings is new. Our proof relies on the partition of unity of affine 1-stresses (Theorem \ref{lm: decomposition of affine dependency}) as well as on, by now, standard tools from the rigidity theory of frameworks. Specifically, we rely on the cone and gluing lemmas \cite[Section 6]{Lee-notes} and on works of Fogelsanger \cite{Fogelsanger88} and Whiteley \cite{Whiteley-84}.

	While we refer the reader to \cite{AsimowRothI,AsimowRothII,Kalai87,Lee-notes} for a thorough introduction to the rigidity theory and all undefined terminology, we briefly summarize some necessary background here. To start, for a $d$-framework $(\Gamma,p)$, we somewhat abuse notation and let $g_2(\Gamma)\coloneqq f_1(\Gamma)-df_0(\Gamma)+\binom{d+1}{2}$. Recall that a $d$-framework $(\Gamma,p)$ that affinely spans $\R^d$ is infinitesimally $d$-rigid if and only if $\dim \Stress^a_2(\Gamma,p)=g_2(\Gamma)$. In particular, for $d\geq 3$, the graph of a simplicial $d$-polytope with its natural embedding is infinitesimally $d$-rigid. This follows from a much more general result of Whiteley \cite{Whiteley-84}. (Of course, this also follows from the $g$-theorem for polytopes.) Furthermore, if $d\geq 3$ and $(\Delta,p)$ is a normal $(d-1)$-pseudomanifold with a generic embedding, then by a result of Fogelsanger \cite{Fogelsanger88},  $(\Delta,p)$ is infinitesimally $d$-rigid. Fogelsanger's result, in fact, applies to a more general class of  {\em simplicial $(d-1)$-circuits} and even more general class of {\em minimal cycles}. A self-contained summary of Fogelsanger's proof is given in \cite[Section 3]{CJT}.

We also recall the statement of the gluing lemma. It asserts that if $(\Gamma,p)$ is a $d$-framework and $\Gamma_1$ and $\Gamma_2$ are subgraphs of $\Gamma$ such that (a) both $(\Gamma_1,p)$ and $(\Gamma_2,p)$ are infinitesimally $d$-rigid, and (b) the $p$-image of $V(\Gamma_1)\cap V(\Gamma_2)$ contains $d$ affinely independent points, then $(\Gamma_1\cup \Gamma_2,p)$ is also infinitesimally $d$-rigid.
	
	The key to our proof of Theorem \ref{thm: i=2} is the following lemma. Some special cases of this lemma are known, see the proof of \cite[Proposition 2.10]{NevoNovinsky}. 
	\begin{lemma}\label{lm: inf rigidity on antistars}
		Let $d$ and $(\Delta, p)$ be as in Theorem \ref{thm: i=2}. Then for every nonempty face $F\in \Delta$, $(\Delta-F, p)$ is an infinitesimally $d$-rigid framework. In particular, $\dim\Stress^a_2(\Delta-F, p)=g_2(\Delta-F)$.
	\end{lemma}
	\begin{proof}
		Throughout this proof, all stars and links are computed in $\Delta$. The complex $\Delta-F$ is a normal pseudomanifold with boundary. First we claim that every minimal interior face of $\Delta-F$ has dimension $\leq d-4$. Indeed, if $\sigma\in\I(\Delta-F)$, then by Lemma \ref{lm:min-int-faces-in-antistars}, there exists $H\subseteq F$ such that $\sigma\cup H$ is a missing face of $\Delta$. Thus, $d-3\geq \dim(\sigma\cup H)>\dim(\sigma)$, and so $\dim\sigma\leq d-4$ as claimed.  This implies that for every $\sigma\in\I(\Delta-F)$, $(\st(\sigma), p)$ is infinitesimally $d$-rigid: in the case that $\Delta$ is a normal pseudomanifold, this follows from Fogelsanger's result \cite{Fogelsanger88} and the cone lemma, while in the case that $\Delta$ is the boundary complex of a polytope, this follows from Whiteley's result \cite{Whiteley-84} and the cone lemma.
		
		Define $K=\bigcup\{\st(\sigma): \sigma\in  \I(\Delta-F)\}$. Clearly, $K\subseteq \Delta -F$. On the other hand, every facet $T$ of $\Delta-F$ is an interior face. Hence $T$ contains a minimal interior face. It follows that $\overline{T}\subseteq K$, and so by purity, $K=\Delta-F$. We conclude that $(\Delta-F, p)$ can be expressed as the union of infinitesimally $d$-rigid frameworks. 
		
		The claim that $(\Delta-F, p)$ is infinitesimally $d$-rigid now follows from repeated applications of the gluing lemma.	Indeed, observe that if $\sigma$ and $\tau$ are elements of $\I(\Delta-F)$ such that $\sigma\cup \tau\in \Delta$, then $\st(\sigma)\cap \st(\tau)\supseteq \st(\sigma\cup \tau)$, where $\st(\sigma\cup \tau)$ contains $d$ vertices of a facet. These $d$ vertices are affinely independent in $\R^d$, and so by the gluing lemma, $(\st(\sigma)\cup \st(\tau), p)$ is infinitesimally $d$-rigid. Thus, to complete the proof, it remains to show that the following graph $\G$ is connected: the vertices of $\G$ correspond to elements of $\I(\Delta-F)$,  and we put an edge between $\sigma$ and $\tau$ if $\sigma\cup \tau\in \Delta$. To see that $\G$ is connected, let $\sigma$ and $\tau$ be elements of $\I(\Delta-F)$. Then there exist facets $H_\sigma$ and $H_\tau$ of $\Delta -F$ that contain $\sigma$ and $\tau$, respectively. Since $\Delta-F$ is strongly connected, we can walk from $H_\sigma$ to $H_\tau$ along a path of facets in $\Delta-F$: $H_\sigma=H^0, H^1, \dots, H^\ell=H_\tau$, such that $H^i\cap H^{i+1}$ is a common ridge of both $H^i$ and $H^{i+1}$. Then $H^i\cap H^{i+1}$ contains a minimal interior face; denote it by $\sigma_{i+1}$. This gives us a sequence $S=(\sigma_0\coloneqq\sigma, \sigma_1, \dots, \sigma_{\ell}, \sigma_{\ell+1}\coloneqq\tau)$ of elements of $\I(\Delta-F)$, where for every $i$, $\sigma_i$ and $\sigma_{i+1}$ are contained in the facet $H^i$. Thus, either $\sigma_i=\sigma_{i+1}$, or $\sigma_i$ and $\sigma_{i+1}$ are connected by an edge in $\G$. In other words, $S$ is a walk from $\sigma$ to $\tau$ in $G$, and so $\G$ is connected. 
	\end{proof}
	
	The proof of Theorem \ref{thm: i=2} now follows in the same spirit as the proof of part 2 of Theorem \ref{lm: partition of unity}.
	
	\smallskip\noindent {\it Proof of Theorem \ref{thm: i=2}: \ }
		It suffices to show that $\{\partial_{x_v}w: w\in \Stress^a_2(\Delta, p), \, v\in V(\Delta)\}=\Stress^a_1(\Delta, p)$. First observe that the sequence $$0\to \Stress^a_2(\Delta-v, p) \to \Stress^a_2(\Delta, p) \stackrel{\partial_{x_v}}{\to} \Stress^a_1(\st(v),p)$$ is exact. Now, whether $\Delta$ is a normal pseudomanifold or the boundary of a polytope, $\dim \Stress^a_2(\Delta, p)=g_2(\Delta)$ and $\dim \Stress^a_1(\st(v), p)=g_1(\lk(v))$. Also, by Lemma \ref{lm: inf rigidity on antistars}, $\dim \Stress^a_2(\Delta-v, p)=g_2(\Delta-v)$. Since $g_2(\Delta)=g_2(\Delta-v)+g_1(\lk(v))$, we conclude that the map $\partial_{x_v}: \Stress^a_2(\Delta, p)\to \Stress^a_1(\st(v),p)$ is onto, and hence that  $\{\partial_{x_v}w: w\in \Stress^a_2(\Delta), \, v\in V(\Delta)\}=\sum_{v\in V(\Delta)}\Stress^a_1(\st(v), p)$. Our claim then follows from Theorem \ref{lm: decomposition of affine dependency}.
		
		The ``in particular" part also follows since the space of affine dependencies of the multiset $p(V(\Delta))$ that affinely spans $\R^d$, determines the multiset itself up to affine equivalence.
	\endproof
	
	We end this section with a corollary to our results. The second part verifies a special case of Conjecture \ref{conj: the support of stresses}.
	
	\begin{corollary}\label{cor: affine 2-stresses}
		Let $d$ and $(\Delta, p)$ be as in Theorem \ref{thm: i=2}. 
		\begin{enumerate}
			\item For any vertex $v\in V(\Delta)$, $g_2(\Delta)\geq g_1(\lk(v))$.
			\item Every edge of $\Delta$ participates in some affine 2-stress on $\Delta$.
		\end{enumerate}
	\end{corollary}
	\begin{proof}
		Part 1 follows from the facts that $g_2(\Delta)=\dim \Stress^a_2(\Delta, p)$, $g_1(\lk(v))=\dim \Stress^a_1(\st(v), p)$, and $\partial_{x_v}: \Stress^a_2(\Delta, p)\to \Stress^a_1(\st(v),p)$ is onto. For part 2, we apply Lemma \ref{lm: inf rigidity on antistars} to an edge $F$. Since the graph of $\Delta$ is the graph of $\Delta-F$ plus the edge $F$, it follows that $\dim \Stress^a_2(\Delta-F,p)=g_2(\Delta-F)=g_2(\Delta)-1$. Hence $F$ must participate in some affine 2-stress on $(\Delta, p)$.
	\end{proof}

	\section{Reconstructing from higher affine stresses}
	In this section we prove several results related to Conjectures \ref{conj:comb-type}, \ref{conj: generalization of Kalai's conjecture}(2), \ref{conj: the support of stresses}, and \ref{conj: structure of stress spaces} for polytopes without large missing faces and for flag PL spheres. We also briefly touch on the $g$-numbers of flag spheres. 
	
	\subsection{Polytopes and flag spheres} 
	We begin with establishing the partition of unity of spaces of linear stresses of antistars.
	\begin{lemma} \label{lm: partition of unity on antistar}
		Let $d\geq 4$. Let $\Delta$ be the boundary complex of a simplicial $d$-polytope $P$ with its natural embedding $p$, or a flag PL $(d-1)$-sphere with a generic embedding $p$. Let $\tau\in \Delta$. Then for $1\leq i\leq d-1$, $\Stress^\ell_i(\Delta-\tau, p)=\sum_{H\in  \I(\Delta-\tau)} \Stress^\ell_i\big(\st(H, \Delta), p\big)$. 
	\end{lemma}
	\begin{proof}
		The case of flag PL spheres follows from \cite[Theorem 50]{AdiprasitoYashfe}. (In this case, $\Delta-\tau$ is a PL ball, its boundary complex is the induced subcomplex of $\Delta-\tau$, and all minimal interior faces of $\Delta-\tau$ are vertices). 
		Thus, assume that $\Delta=\partial P$. The proof of this case is essentially the same as that of \cite[Theorem 16]{Lee96}. Since $P$ is a polytope, there is a line shelling of $\Delta$ that lists the facets of the star of $\tau$ last. Consequently, there exists a shelling of $\Delta-\tau$. Consider such a shelling and let $F_1, \dots, F_k$ be the facets at the shelling steps of type $i$; here $k=h_i(\Delta -\tau)$. For $1\leq j\leq k$, let $G_j=F_j\backslash r(F_j)$ and let $\Delta_j$ be the subcomplex of $\Delta-\tau$ generated by $F_j$ and all the facets that were added before $F_j$ (in the shelling order). We claim that $\st(G_j, \Delta)$ is a subcomplex of $\Delta_j$; in particular, $G_j$ is an interior face of $\Delta-\tau$. Indeed, $\lk(G_j,\Delta)$ is an $(i-1)$-sphere and in the induced shelling of this sphere, the step that adds $F_j$ corresponds to the shelling step of type $i$ (that adds $r(F_j)$). This means that after this step, all facets of the link are in the complex. Now, since $\st(G_j,\Delta)$ and $\st(G_j,\Delta)-r(F_j)$ are Cohen--Macaulay complexes with $h_i(\st(G_j,\Delta))=1$ and $h_i\big(\st(G_j,\Delta)-r(F_j)\big)=0$, it follows that there is
		a  linear $i$-stress $\omega_j$ supported on $\st(G_j,\Delta)\subseteq \Delta_j$ that assigns a nonzero weight to $r(F_j)$. Also, since $\omega_j$ is supported on $\Delta_j$, it assigns zero weights to all $r(F_s)$ with $s>j$. We conclude that  $\{\omega_j: 1\leq j\leq k\}$ is a linearly independent set of stresses. Furthermore, since $k=h_i(\Delta-\tau)$, this set is a basis of $\Stress^\ell_i(\Delta-\tau,p)$. The result follows since $G_j$ is an interior face of $\Delta-\tau$ and hence $G_j$ contains a minimal interior face $H_j$ of $\Delta-\tau$. Therefore,
		$\Stress^\ell_i(\Delta-\tau, p)=  \sum_j \Stress^\ell_i\big(\st(G_j, \Delta), p\big) = \sum_{H\in \I(\Delta-\tau)} \Stress^\ell_i\big(\st(H, \Delta), p\big)$.
	\end{proof}

	Lemma \ref{lm: partition of unity on antistar} along with Theorem \ref{thm: Lefschetz}  imply the following higher-dimensional analogs of Lemma \ref{lm: inf rigidity on antistars} and part 1 of Corollary \ref{cor: affine 2-stresses}. This result will be used in essentially all proofs of this section.
	
	\begin{lemma}\label{cor: surjection of affine stress spaces}
	    Let $d\geq 4$ and $1\leq j\leq i\leq d/2$. Let $\Delta$ be the boundary complex of a simplicial $d$-polytope $P$ with its natural embedding $p$, or a flag PL $(d-1)$-sphere with a generic embedding $p$. In the case that $\Delta=\partial P$, assume further that all missing faces of $\Delta$ have dimension $\leq d-2i+1$. Then for any nonempty face $\tau$ of $\Delta$ of dimension $\leq j-1$, 
	    \begin{enumerate}
	        \item the map $\partial_c: \Stress^\ell_j(\Delta-\tau, p) \to \Stress^\ell_{j-1}(\Delta-\tau,p)$ is onto. In particular, $\dim \Stress^a_j(\Delta-\tau,p)=g_j(\Delta-\tau)$;
	        \item the map $\partial_{x_\tau}: \Stress^a_j(\Delta,p)\to \Stress^a_{j-|\tau|}(\st (\tau,\Delta), p)$ is onto, and hence $g_j(\Delta)\geq g_{j-|\tau|}(\lk(\tau))$.
	    \end{enumerate}
	\end{lemma}
	\begin{proof}
	    Throughout the proof, all stars and links are computed in $\Delta$. Let $\sigma$ be a minimal interior face of $\Delta-\tau$. Since all missing faces of $\Delta$ have dimension $\leq d-2i+1$, it follows from Lemma \ref{lm:min-int-faces-in-antistars} that $\dim\sigma\leq d-2i$, and so $\lk(\sigma)$ is a PL sphere of dimension at least $2i-2$. By the cone lemma and by Theorem \ref{thm: Lefschetz}, we obtain that the map $\partial_c: \Stress^\ell_j(\st(\sigma),p)\to \Stress^\ell_{j-1}(\st(\sigma),p)$ is onto for all $j\leq i$ and all minimal interior faces $\sigma$ of $\Delta-\tau$. This together with Lemma \ref{lm: partition of unity on antistar} implies that $\partial_c: \Stress^\ell_j(\Delta-\tau,p)\to \Stress^\ell_{j-1}(\Delta-\tau,p)$ is also onto. Finally, since $\Delta-\tau$ is Cohen--Macaulay, $\dim \Stress^\ell_j(\Delta-\tau)=h_j(\Delta-\tau)$ for all $j$. We conclude that $\dim \Stress^a_j(\Delta-\tau)=g_j(\Delta-\tau)$ for all $j\leq i$.

		As in the proof of Theorem \ref{lm: partition of unity}, we have the following exact sequence: 
		$$0\to \Stress^a_j(\Delta-\tau,p) \to \Stress^a_j(\Delta,p) \stackrel{\partial x_\tau}{\longrightarrow} \Stress^a_{j-|\tau|}(\st (\tau),p).$$
		By part 1 and Theorem \ref{thm: Lefschetz}, the dimensions of the three spaces in the sequence are $g_j(\Delta-\tau)$, $g_j(\Delta)$, and $g_{j-1}(\lk(\tau))$, respectively. Since for any sphere $\Delta$ and $\tau\in\Delta$, $g_j(\Delta)=g_j(\Delta-\tau)+g_{j-|\tau|}(\lk (\tau))$, it follows that the right-most map in this sequence, $\partial_{x_\tau}: \Stress^a_j(\Delta, p)\to \Stress^a_{j-|\tau|}(\st (\tau), p)$, must be onto. Consequently, $g_j(\Delta)\geq g_{j-|\tau|}(\lk(\tau))$. This completes the proof of both parts.
	\end{proof}
	
    With Lemmas  \ref{lm: partition of unity on antistar} and \ref{cor: surjection of affine stress spaces} at our disposal, we are ready to verify a special case of Conjecture \ref{conj: generalization of Kalai's conjecture} and a special case of Conjecture \ref{conj: structure of stress spaces}. 
	\begin{theorem}\label{thm: main}
Let $d\geq 4$ and $1\leq j < i\leq d/2$. Let $\Delta$ be the boundary complex of a simplicial $d$-polytope $P$ with its natural embedding $p$, or a flag PL $(d-1)$-sphere with a generic embedding $p$. In the case that $\Delta=\partial P$, assume further that all missing faces of $\Delta$ have dimension $\leq d-2i+1$. Then $$\Span\{\partial_{\mu} \omega: \omega\in\Stress^a_i(\Delta,p), \, \mu\in \M_{i-j}(V(\Delta))\}=\Stress^a_j(\Delta,p),$$ and so $\Stress^a_i(\Delta, p)$ determines $\Stress^a_j(\Delta, p)$. In particular, $\Stress_i^a(\Delta,p)$ determines the positions of vertices of $\Delta$ up to affne equivalence.
	\end{theorem}
	
	\begin{proof}
	    By Lemma \ref{cor: surjection of affine stress spaces}, for any $(i-j-1)$-face $\tau$ of $\Delta$, the map $\partial_{x_\tau}: \Stress^a_i(\Delta,p) \to \Stress^a_j(\st(\tau), p)$ is onto.
	    The statements then follow from the partition of unity of the space $\Stress^a_j(\Delta,p)$; see Theorem \ref{lm: general partition of unity, PL} and Theorem \ref{thm: general partition of unity}.
	\end{proof}

	Another application of Lemma \ref{cor: surjection of affine stress spaces} allows us to also establish the following special case of Conjecture \ref{conj: the support of stresses} (cf.~Corollary \ref{cor: affine 2-stresses}).
	\begin{corollary}  \label{cor:the support of stresses}
    	Let $2\leq i\leq d/2$. Let $\Delta$ be the boundary complex of a simplicial $d$-polytope with its natural embedding $p$, or a flag PL $(d-1)$-sphere with a generic embedding $p$. In the case that $\Delta=\partial P$, assume also that all missing faces of $\Delta$ have dimension $\leq d-2i+1$. Then every $(i-1)$-face of $\Delta$ participates in some affine $i$-stress on $\Delta$. 
    \end{corollary}
    \begin{proof}
    	Let $\tau$ be an $(i-1)$-face of $\Delta$. By part 2 of Lemma \ref{cor: surjection of affine stress spaces}, the map $\partial x_\tau: \Stress^a_i(\Delta, p)\to \Stress^a_0(\st(\tau,\Delta), p)\cong \R$ is onto. Any preimage of $1$ is then an affine $i$-stress that has $\tau$ in its support.
    \end{proof}

	It is worth remarking that the proof of Lemma \ref{lm: partition of unity on antistar} relies on \cite[Theorem 50]{AdiprasitoYashfe}. This theorem holds for all simplicial balls whose boundary complex is an induced subcomplex. If it continues to hold for all {\em homology} rather than just simplicial balls,  then all results in this subsection as well as Theorem \ref{thm: flag LB g-number} from Section 6.2 would hold for all flag $\Z/2\Z$-homology spheres rather than just flag PL spheres.

	\subsection{An interlude: $g$-numbers of flag spheres}
	The techniques developed in the previous subsection will be useful in obtaining lower bounds. The {\em octahedral $(d-1)$-sphere} is the boundary complex of the $d$-dimensional cross-polytope $\C^*_d$. As an abstract simplicial complex, $\partial \C^*_d$ is the join of $d$ copies of the $0$-sphere. In particular, octahedral spheres are flag and $h_i(\partial\C^*_d)=\binom{d}{i}$ for all $0\leq i\leq d$. Meshulam \cite{Mesh} proved that in the class of all flag $(d-1)$-spheres, the octahedral sphere simultaneously minimizes all the $f$-numbers. This result was strengthened by Athanasiadis \cite{Athan} who showed that, in fact, it simultaneously minimizes all the $h$-numbers. Here we prove that in the class of flag PL spheres, it even simultaneously minimizes all the $g$-numbers. This was conjectured in \cite{Z-flagsurvey} where the case of $i=2$ was established.

	\begin{theorem}\label{thm: flag LB g-number}
		Let $\Delta$ be a flag PL $(d-1)$-sphere. Then for every $1\leq i\leq d/2$, $g_i(\Delta)\geq \binom{d}{i}-\binom{d}{i-1}$, and equality holds if and only if $\Delta$ is the octahedral sphere.
    \end{theorem}
    \begin{proof}
		If $\Delta=\partial\C^*_d$, then $g_i(\Delta)=\binom{d}{i}-\binom{d}{i-1}$ for all $1\leq i\leq d/2$. To prove the inequality and show that $\partial\C^*_d$ is the only minimizer, we use  induction on $d$. The claim is known to hold for $i=1$. Thus assume that $i>1$ and let $v$ be a vertex of $\Delta$. In particular, $\lk(v)$ is a flag PL $(d-2)$-sphere. If $d=2i$, then by part 2 of Lemma \ref{cor: surjection of affine stress spaces} and the inductive hypothesis, 
    	\begin{equation*}
    		\begin{split}
    			&g_i(\Delta)\geq g_{i-1}(\lk(v))
    			\geq\binom{2i-1}{i-1}-\binom{2i-1}{i-2}\\
    			=\ &\binom{2i-1}{i-1}+\binom{2i-1}{i}-\binom{2i-1}{i-1}-\binom{2i-1}{i-2}=\binom{2i}{i}-\binom{2i}{i-1}.
    		\end{split}
    	\end{equation*} If $d>2i$, then let $v'$ be an interior vertex of $\Delta-v$ (it exists since $\Delta$ is flag). Since every affine $i$-stress on $\st(v')$ is also an affine $i$-stress on $\Delta-v$, by Lemma \ref{cor: surjection of affine stress spaces} and the inductive hypothesis, 
			\begin{equation*}
		\begin{split}
			& g_i(\Delta)=g_i(\Delta-v)+g_{i-1}(\lk(v))\geq g_i(\lk(v'))+g_{i-1}(\lk(v))\\
			\geq \ & \binom{d-1}{i}-\binom{d-1}{i-1}+\binom{d-1}{i-1}-\binom{d-1}{i-2}=\binom{d}{i}-\binom{d}{i-1}.
		\end{split}
	    \end{equation*}
			In both cases, if equality holds, then the link of every vertex $v$ is octahedral. Thus, every vertex $v$ has degree $2d-2$, and so $h_2(\Delta)=f_1(\Delta)-(d-1)f_0(\Delta)+\binom{d}{2}=\binom{d}{2}$. Since $g_2(\Delta)\geq \binom{d}{2}-\binom{d}{1}$, it follows that $f_0(\Delta)\leq 2d$. As $\Delta$ is flag, we must have $f_0(\Delta)=2d$, and so $\Delta$ itself is octahedral.
    \end{proof}

    \subsection{Sign vectors of affine stresses}
    In this subsection we discuss another conjecture related to Kalai's Conjecture \ref{conj:comb-type}. The statement of this conjecture is based on the notion of sign vectors.
    \begin{definition} Let $\Delta$ be the boundary complex of a simplicial $d$-polytope $P$ with its natural embedding $p$. For an affine $i$-stress $\lambda$ on $(\Delta, p)$ and an $(i-1)$-face $G$ of $\Delta$, let $$\sign(\lambda_G)=
    	\begin{cases}
    		+ & \mbox{if }\lambda_G>0 \\
    		- & \mbox{if }\lambda_G<0 \\
    		0 & \mbox{if }\lambda_G=0.
    	\end{cases}$$
    	Define $\V_i(P)=\{(\sign(\lambda_G))_{G\in \Delta, \;|G|=i}: \lambda\in \Stress^a_i(\Delta,p)\}$. Thus $\V_i(P)$ is the collection of sign vectors of the squarefree parts of $i$-stresses on $P$.
    \end{definition}
		In view of results from \cite{N-Z22}, the following strengthening of Conjecture \ref{conj:comb-type} was proposed there:
   
    \begin{conjecture}\label{conj: sign vectors}
    	Let $2\leq i\leq d/2$. Let $P\subset \R^d$ be a simplicial $d$-polytope. The $(i-1)$-skeleton of $\partial P$ and the set $\V_i(P)$ determine the combinatorial type of $P$ (i.e., they determine the entire abstract simplicial complex $\partial P$.)
    \end{conjecture}
    
    The goal of this section is to verify Conjecture  \ref{conj: sign vectors} in the following special case. 
		\begin{theorem}  \label{thm:comb-type}
    	Let $2\leq i\leq d/2$. Let $P\subset \R^d$ be a simplicial $d$-polytope whose boundary complex has only missing faces of dimension $\leq d-2i+1$. Then the $(i-1)$-skeleton of $\partial P$ and the set $\V_i(P)$ determine the entire complex $\partial P$.
    \end{theorem}
  
	The following lemma will be handy.
		\begin{lemma} \label{lm:signs}
		Let $2\leq i\leq d/2$. Let $P$ be a simplicial $d$-polytope, let $F$ be a $(j-1)$-face of $\partial P$ with $j\leq i-1$, and let $Q$ be the quotient polytope $P/F$. 
		Assume also that all missing faces of $\partial P$ have dimension $\leq d-2i+1$. Then every affine $(i-j)$-stress $\omega'$ on $Q$ can be lifted to an affine $i$-stress $\omega$ on $P$ with the following property: for each $(i-j-1)$-face $\tau$ of $\partial Q$, $\sign(\omega'_\tau)=\sign(\omega_{F\cup\tau})$.
    \end{lemma} 
    \begin{proof}
		We work with the boundary complex of $P$, $\partial P$, with its natural embedding $p$, and the boundary complex of $Q$, $\partial Q$, with its natural embedding $q$; in particular, $\partial Q=\lk(F,\partial P)$.
		Consider the sequence 
		$$\Stress^a_i(\partial P, p) \stackrel{\partial_{x_{F}}}{\longrightarrow} \Stress^a_{i-j}\big(\st(F,\partial P), p\big) \stackrel{\phi_{i-j}}{\longrightarrow} \Stress^a_{i-j} \big(\lk(F,\partial P), q\big), $$
		where $\phi_{i-j}$ is the map from Lemma \ref{cone-lemma1}. The map $\partial_{x_F}$ is surjective by Lemma \ref{cor: surjection of affine stress spaces}, while the map $\phi_{i-j}$ is an isomorphism by Lemma \ref{cone-lemma1}. Furthermore, by the remark following Lemma \ref{cone-lemma1},  if $\omega'$ is an affine  $(i-j)$-stress on $(\lk(F,\partial P),q)$ and $\omega''\coloneqq(\phi_{i-j})^{-1}(\omega')$, then for every $(i-j-1)$-face $\tau$ of $\lk(F, \partial P)$, $\omega'_\tau$ and $\omega''_{\tau}$ have the same signs. The result follows by letting $\omega$ be any element of $\Stress^a_i(\partial P,p)$ such that $\partial_{x_F}(\omega)=\omega''$ and noting that $\omega_{F\cup\tau}= \omega''_{\tau}$.
    \end{proof}
		
		By Lemma 4.5 in \cite{N-Z22}, to prove Theorem \ref{thm:comb-type}, it suffices to establish the following result, which is interesting in its own right. This result concludes this section.
		
    \begin{theorem}\label{thm: sign vectors under lifting}
    	Let $i\geq 1$ and $d\leq 2i$. Let $\partial P$ be the boundary complex of a simplicial $d$-polytope $P$ with its natural embedding $p$. Assume that all missing faces of $\partial P$ have dimension $\leq d-2i+1$. Let $M$ be a missing face of $\partial P$ of size $\geq i+1$ and let $F\subset M$ be any subset of size $i-1$. Then there exists an affine $i$-stress $\lambda$ on $(\partial P,p)$ with the following property: for every $(i-1)$-face $G=F\cup v$ of $\partial P$, $\lambda_G>0$ if $v\in M\backslash F$ while $\lambda_G\leq 0$ if $v \notin M$.
    \end{theorem}
    \begin{proof}
    If $i=1$, then $d-2i+1<d$, so $P$ is a nonsimplex polytope of dimension $d\geq 2$ and $F=\emptyset$. Since $M$ is a missing face, it follows that the intersection  $\conv(p(M))\cap \conv(V(P)\backslash p(M))$ is nonempty and that it is contained in the relative interior of $\conv(p(M))$. Thus, there exists $\lambda\in \Stress^a_1(\partial P,p)$ such that $\lambda_v>0$ if the vertex $v$ is in $M$ and $\lambda_v\leq 0$ if $v\notin M$. This completes the proof of the $i=1$ case.
    	
    We now prove the statement for $i>1$. Let $Q\coloneqq P/F$ be the quotient polytope and let $q$ be the natural embedding of $\partial Q$. Then $Q$ is a $(d-i+1)$-polytope and $M'\coloneqq M\backslash F$ is a missing face of $\partial Q$.
		Since all missing faces of $\partial P$ have dimension $\leq d-2i+1<d-i$, so do all missing faces of $\partial Q$. In particular, $Q$ is a nonsimplex polytope. Applying the first paragraph to the triple $(Q,M',\emptyset)$, we find an affine stress $\lambda'\in\Stress^a_1(\partial Q,q)=\Stress^a_1(\lk(F,\partial P),q)$ such that for every vertex $v$ of $\partial Q$, $\lambda'_v>0$ if $v\in M'$ and $\lambda'_v\leq 0$ otherwise. Lemma \ref{lm:signs} then guarantees the existence of a stress $\lambda\in\Stress^a_i(\partial P,p)$ such that for every $(i-1)$-face $F\cup v$ of $\partial P$, $\lambda_{F\cup v}>0$ if $v \in M'=M\backslash F$ and $\lambda_{F\cup v}\leq 0$ if $v\notin M$. This completes the proof.
		\end{proof}

	\section{$k$-stacked spheres}
To close the paper, we prove Conjecture \ref{conj: generalization of Kalai's conjecture}(2) for the case of $k$-stacked polytopes and spheres. 

Let $0\leq k\leq d$. An $\R$-homology $d$-ball $B$ is called {\em $k$-stacked} if all interior faces of $B$ are of dimension $\geq d-k$; in other words, $B$ is $k$-stacked if every face of $B$ of dimension $\leq d-k-1$ is a face of $\partial B$. A simplicial $(d-1)$-sphere $\Delta$ is {\em $k$-stacked} if there exists a $k$-stacked $\R$-homology $d$-ball $B$ such that $\Delta=\partial B$; in this case, $\skel_{d-k-1}(\Delta)=\skel_{d-k-1}(B)$. We say that a simplicial polytope $P$ is $k$-stacked if $\partial P$ is a $k$-stacked sphere. The significance of $k$-stacked spheres is explained by the Generalized Lower Bound Theorem: if $0\leq k\leq d/2-1$, then a simplicial $(d-1)$-sphere $\Delta$ satisfies $g_{k+1}(\Delta)=0$ if and only if $\Delta$ is $k$-stacked. This result for the boundary complexes of simplicial polytopes is due to Murai and Nevo \cite{MuraiNevo2013} (see also \cite{Adiprasito-toric}); the general case follows from Murai--Nevo's results and the $g$-theorem (see Theorem \ref{thm: Lefschetz}).

Murai and Nevo also proved that if  $0\leq k\leq d/2-1$ and $\Delta$ is a $k$-stacked $(d-1)$-sphere, then a $k$-stacked $\R$-homology $d$-ball whose boundary is equal to $\Delta$ is {\em unique}.  This ball is given by $$T(\Delta)\coloneqq \{F\subseteq V(\Delta): \skel_{d-k-1}(\overline{F})\subseteq \Delta\};$$ see \cite[Theorem 2.3]{MuraiNevo2013}. Furthermore, if $P$ is a $k$-stacked $d$-polytope then $T(\partial P)$ provides a {\em geometric triangulation} of $P$; see  \cite[Theorem 1.2]{MuraiNevo2013}. In particular, the $p$-images of vertices of any $d$-face of $T(\partial P)$ are affinely independent.

Assume that $0\leq k\leq d/2-1$ and that $\Delta$ is a $k$-stacked $(d-1)$-sphere with a $d$-embedding $p$. The complex $T(\Delta)$ is $d$-dimensional; hence, to talk about stress spaces of $T(\Delta)$, we need to specify a map $\tilde{p} : V(T(\Delta))=V(\Delta) \to \R^{d+1}$. We define such $\tilde{p}$ by $\tilde{p}(v)\coloneqq(p(v),1)$. The important thing to notice is that since $\skel_{d-k-1}(T(\Delta))=\skel_{d-k-1}(\Delta)$, it follows from the definition of $\tilde{p}$ that $\Stress^a_j(\Delta,p)=\Stress^\ell_j(T(\Delta),\tilde{p})$ for all $j\leq d/2$.

Another thing to notice is that by definition of $T(\Delta)$, $\mathcal{I}(T(\Delta))$ is precisely the set of missing faces of $\Delta$ of dimension $\geq d-k$ and 
$$ T(\Delta)=\bigcup_{F\in \mathcal{I}(T(\Delta))} \st(F, T(\Delta)) =\bigcup_{F\in \mathcal{I}(T(\Delta))}  \overline{F}*S_F, $$
where $S(F)=\lk(F,T(\Delta))$ is an $\R$-homology sphere of dimension $d-|F|\leq k-1$; in particular, $S_F\subseteq \Delta$.

We are now ready to prove the following case of Conjecture \ref{conj: generalization of Kalai's conjecture}. 
\begin{theorem} \label{thm: k-stacked polytopes}
Let $1\leq i\leq k\leq d/2-1$. Let $\Delta$ be a $k$-stacked $(d-1)$-sphere that has no missing faces of dimension $\geq d-i+1$. Let $p$ be a $d$-embedding of $\Delta$ such that the $p$-images of vertices of any $d$-face of $T(\Delta)$ are affinely independent. Then $\Stress^a_1(\Delta,p)=\Span\left\{\partial_{\mu} \omega: \omega\in \Stress^a_i(\Delta, p), \, \mu\in \M_{i-1}(V(\Delta))\right\}$, and so $\Stress^a_i(\Delta, p)$ determines $\Stress^a_1(\Delta, p)$. In particular, if $P$ is a $k$-stacked $d$-polytope, then the space of affine $i$-stresses of $P$ determines $P$ up to affine equivalence.
\end{theorem}
\proof
Since $\Stress^a_j(\Delta,p)=\Stress^\ell_j(T(\Delta),\tilde{p})$ for all $j\leq d/2$, it suffices to show that  $\Stress^\ell_i(T(\Delta), \tilde{p})$ determines $\Stress^\ell_1(T(\Delta), \tilde{p})$. Also, since  the $p$-images of vertices of any $d$-face of $T(\Delta)$ are affinely independent, the $\tilde{p}$-images of these vertices are linearly independent. Finally, by our assumptions, the dimension of each $F$ appearing in the decomposition $T(\Delta)=\bigcup_{F\in I(T(\Delta))}  \overline{F}*S_F$ is $\leq d-i$,  and so each homology sphere $S_F$ in this decomposition has dimension $\geq i-1$. Thus Remark \ref{rm: linear} applies to $T(\Delta)$ while Corollary \ref{cor:simplex-join-sphere} applies to each $\overline{F}*S_F$. We obtain that
\begin{eqnarray*}
&&\Span\left\{\partial_{\mu} \omega: \omega\in \Stress^\ell_i(T(\Delta),\tilde{p}), \, \mu\in \mathcal{M}_{i-1}(V(\Delta))\right\} \\
 &\supseteq&
 \sum_{F\in \mathcal{I}(T(\Delta))} \Span\left\{\partial_{x_G} \omega: \omega\in \Stress^\ell_i(\overline{F}*S_F,\tilde{p}), \,G\in S_F, \, |G|=i-1\right\} \\
&\stackrel{(\ast)}{=}& \sum_{F\in \mathcal{I}(T(\Delta))} \Stress^\ell_1\left(\overline{F}*S_F,\tilde{p}\right) \\
& = & \sum_{F\in \mathcal{I}(T(\Delta))} \Stress^\ell_1\left(\st(F, T(\Delta)),\tilde{p}\right)  \stackrel{(\dag)}{=} \Stress^\ell_1 \left(T(\Delta),\tilde{p}\right).
\end{eqnarray*} 
Here $(\ast)$ is by Corollary \ref{cor:simplex-join-sphere} and $(\dag)$ is by Remark \ref{rm: linear}. The result follows.
\endproof

\subsection*{Acknowledgments} We are grateful to Satoshi Murai for his interest in and comments on the paper and to Gil Kalai for sharing with us his new reconstruction conjectures related to non-simplicial polytopes. We also thank Karim Adiprasito anf Geva Yashfe for inspiring conversations, and the anonymous referee for helpful suggestions on improving the presentation.

\subsection*{Data availability} Data sharing not applicable to this article as no datasets were generated or analysed during the current study.

	{\small
		\bibliography{refs}
		\bibliographystyle{plain}
	}

	\appendix
	\section{Proof of Conjecture  \ref{conj: general partition of unity}}
	The goal of this Appendix is to sketch the proof of Conjecture \ref{conj: general partition of unity}.
	\begin{theorem}\label{thm: general partition of unity}
		Let $(\Delta,p)$ be either the boundary complex of a simplicial $d$-polytope with its natural embedding $p$, or a $\Z/2\Z$-homology $(d-1)$-sphere with a generic embedding $p$, and let $i$ be a natural number such that $i\leq \lfloor (d-1)/2\rfloor$. Then $\Stress^a_i(\Delta,p)=\sum_{v\in V(\Delta)} \Stress^a_i(\st(v), p)$. 
	\end{theorem}
	
	The proof uses the (dual) language of the Stanley--Reisner rings. Specifically, we denote by $\R[\Delta]$ the Stanley--Reisner ring of $\Delta$ and, for a face $\tau$ of $\Delta$, we denote by $\R[\st (\tau)]$ the Stanley--Reisner ring of the star of $\tau$ in $\Delta$, considered as an $\R[\Delta]$-module. As in Section 3.1, we let $\Theta=\Theta(p)$ be the collection of $d+1$ linear forms $(\theta_1,\ldots,\theta_d,\theta_{d+1})$ in $\R[X]$ determined by $p$. In particular, $\theta_{d+1}=\sum_{v\in V} x_v$.  Finally, for a graded $\R[\Delta]$-module $M$ and any integer $j$, we denote by $M_j$ the $j$-th graded component of $M$. Some computations below rely on a simple observation that $\R[\st (\tau)]_j=0$ for all $j<0$. 
	
	The statement that $\Stress^a_i(\Delta,p)=\sum_{v\in V(\Delta)} \Stress^a_i(\st(v), p)$ is  easily seen to be equivalent to the statement that the map $\big(\R[\Delta]/(\Theta)\big)_i \to \sum_{v\in V(\Delta)} \big(\R[\st(v)]/\Theta\R[\st(v)]\big)_i$, induced by natural surjections $\R[\Delta] \to \R[\st(v)]$, is injective. To establish this result we adapt the proof of Theorem 33 from \cite{AdiprasitoYashfe}.  The key new idea is to look at the Koszul complex $K^*(\Theta)$ w.r.t.~all $d+1$ elements $(\theta_1,\ldots,\theta_d, \theta_{d+1})$ of $\Theta(p)$ rather than just w.r.t.~the first $d$ elements. Our assumptions on $(\Delta,p)$ along with the $g$-theorem (see \cite{Stanley80,McMullen96} for the case of polytopes and \cite[Theorem 1.3]{KaruXiao} for the case of spheres) imply that $\R[\Delta]$ is a Cohen--Macaulay ring of Krull dimension $d$, that the sequence $\theta_1,\ldots,\theta_d$ is a regular sequence on $\R[\Delta]$, and that the map $\cdot \theta_{d+1}: \big(\R[\Delta]/(\theta_1,\ldots,\theta_d)\big)_{j-1} \to  \big(\R[\Delta]/(\theta_1,\ldots,\theta_d)\big)_{j}$ is injective for all  $j\leq \lceil d/2\rceil$.  Similar statements apply to $\R[\st(\tau)]$ for any face $\tau$ of $\Delta$. Specifically,  $\theta_1,\ldots,\theta_d$ is a regular sequence on $\R[\st (\tau)]$  and  
	\[\cdot \theta_{d+1}: \big(\R[\st (\tau)]/(\theta_1,\ldots,\theta_d)\R[\st (\tau)]\big)_{j-1} \to  \big(\R[\st (\tau)]/(\theta_1,\ldots,\theta_d)\R[\st (\tau)]\big)_{j} \]
is injective for all $j\leq \lceil (d-|\tau|)/2\rceil$.

The above paragraph and standard results about Koszul complexes (such as Theorem 21 in \cite{AdiprasitoYashfe}) imply that for any face $\tau$ and any integer $j$, the following complex of vector spaces over $\R$
\begin{eqnarray*} 0 \to \R[\st(\tau)]_{j-d-1} \otimes K^0(\Theta) &\stackrel{\partial^0}{\to}& \R[\st(\tau)]_{j-d} \otimes K^1(\Theta) \stackrel{\partial^1}{\to}\cdots \\
\to\R[\st(\tau)]_{j-1} \otimes K^d(\Theta) &\stackrel{\partial^d}{\to}& \R[\st(\tau)]_{j}  \otimes K^{d+1}(\Theta) 
 \to \big(\R[\st (\tau)]/\Theta\R[\st (\tau)]\big)_j \to 0
\end{eqnarray*}
is almost exact, namely, 
\begin{enumerate}
\item[(*)] all cohomologies of this complex, except possibly for $H^d$, vanish, and
\item[(**)]  if $j\leq \lceil (d-|\tau|)/2\rceil$, then $H^d$ also vanishes.
\end{enumerate}
 
We now proceed as in the proof of \cite[Theorem 33]{AdiprasitoYashfe}. Let $\Delta^{(j)}$ denote the set of $j$-faces of $\Delta$. (In particular, $\Delta^{(0)}=V(\Delta)$.) Let $P^*=P^*(\Delta)$ be the partition complex
\[ 0\to \R[\Delta]\to \bigoplus_{v\in\Delta^{(0)}} \R[\st(v)] \to \cdots \to \bigoplus_{\sigma\in \Delta^{(d-1)}} \R[\st(\sigma)] \to 0\]
with indexing such that $P^{-1}=\R[\Delta]$ and $P^j=\bigoplus_{\tau\in \Delta^{(j)}} \R[\st(\tau)]$ for $j\geq 0$.
Let $\tilde{K}^*(\Theta)$ be the augmented Koszul complex w.r.t.~$\Theta=\Theta(p)$, i.e., the complex
\[ K^0(\Theta) \to \cdots \to K^{d+1}(\Theta) \to \tilde{K}^{d+2}(\Theta)\coloneqq\R[\Delta]/(\Theta) \to 0.\]
Finally, let $C^{*,*}$ be the double complex $P^*\otimes\tilde{K}^*(\Theta)$ endowed with the grading defined in \cite[Section 5.1.1]{AdiprasitoYashfe}.

We fix $i\leq \lfloor (d-1)/2\rfloor=\lceil(d-2)/2\rceil$ and consider the $i$-th graded piece of $C^{*,*}$, $C^{*,*}_{\, i}$. The outline of the proof is as follows. Properties (*) and (**) above along with \cite[Lemma 8]{AdiprasitoYashfe} applied in the vertical direction of $C^{*,*}_{\, i}$ imply that $H^d\big(\Tot\big(C^{*,*}_{\, i}\big)^*\big)=H^{d+1}\big(\Tot\big(C^{*,*}_{\, i}\big)^*\big)=0$. This in turn implies that $H^d\big(\Tot\big(C^{*,*\leq d+1}_{\, i}\big)^*\big)$ is isomorphic to $H^{d+1}\big(C^{*-d-2,d+2}_{\, i}\big)$. It then follows that the kernel of the map $\big(\R[\Delta]/(\Theta)\big)_i \to \sum_{v\in V(\Delta)} \big(\R[\st(v)]/\Theta\R[\st(v)]\big)_i$ is isomorphic to  $H^d\big(\Tot\big(C^{*,*\leq d+1}_{\, i}\big)^*\big)$. Now, by \cite[Proposition 26]{AdiprasitoYashfe}, $P^*_t$ is exact for all $t\neq 0$ and since $\Delta$ is Cohen--Macaulay, the only nontrivial cohomology that $P^*_0$ has is $H^{d-1}$. Finally, applying \cite[Lemma 8]{AdiprasitoYashfe} in the horizontal direction of $C^{*,*\leq d+1}_{\, i}$, we conclude that $H^d\big(\Tot\big(C^{*,*\leq d+1}_i\big)^*\big)=0$, and so the map $\big(\R[\Delta]/(\Theta)\big)_i \to \sum_{v\in V(\Delta)} \big(\R[\st(v)]/\Theta\R[\st(v)]\big)_i$ is injective. This completes the proof.
\endproof

\end{document}